\documentclass[largeformat]{interact}
\usepackage{epstopdf}
\usepackage[caption=false]{subfig}
\usepackage[utf8]{inputenc}
\usepackage{amssymb}
\usepackage{amsmath}
\usepackage{amsthm}
\usepackage{hyperref}
\usepackage{breqn}
\usepackage[dvipsnames,svgnames,table]{xcolor}
\usepackage{mathtools}
\usepackage{bm}
\usepackage{enumitem}
\usepackage{mathrsfs}
\usepackage{graphicx}
\usepackage{caption}
\usepackage{url}
\usepackage{placeins}
\usepackage[english]{babel}
\usepackage{array}
\usepackage{booktabs}
\usepackage{lineno}
\usepackage[numbers,sort&compress]{natbib}
\bibpunct[, ]{[}{]}{,}{n}{,}{,}
\makeatletter
\def\NAT@def@citea{\def\@citea{\NAT@separator}}
\makeatother

\theoremstyle{plain}
\newtheorem{theorem}{Theorem}[section]
\newtheorem{lemma}[theorem]{Lemma}
\newtheorem{corollary}[theorem]{Corollary}
 \newtheorem{proposition}[theorem]{Proposition}
\newcommand\normx[1]{\left\Vert#1\right\Vert}
\newcommand{\pr}[1]{{\partial{#1}}}
\theoremstyle{definition}
\newtheorem{definition}[theorem]{Definition}
\newtheorem{example}[theorem]{Example}

\theoremstyle{remark}
\newtheorem{remark}{Remark}

\def\bmatrix#1{\left[\begin{matrix}
		#1
	\end{matrix}\right]}
\def \diag{\mathrm{diag}}
\def \R{{\mathbb R}}

\def \C{{\mathscr{C}}}
\def \D{\Delta}
\def \E{{\mathbf E}}
\def \F{{\mathbf F}}
\def \L{{\mathrm L}}
\def \U{{\mathrm U}}
\def \Da{{\mathrm D}}
\def \sign{\mathrm{sign}}
\def \rank{\mathrm{rank}}
\def \ep{{\varepsilon}}
\def \M{{\mathscr{M}}}
\def \a{{\bf a}}

\def \b{{\bf b}}
\def \d{{\bf d}}
\def \e{{\bf e}}
\def \f{{\bf f}}
\def \g{{\bf g}}
\def \h{{\bf h}}
\def \r{{\bf r}}
\def \S{\mathcal{S}}
\def \vr{{\Bigl\vert }}
\def \Vr{{\Bigl\Vert }}
\def \x{{\bf x}}
\newcommand{\dm}[1]{{\displaystyle{#1}}}
\begin{document}


\title{Condition numbers for the Moore-Penrose inverse and the least squares problem involving rank-structured matrices}

\author{
\name{Sk. Safique Ahmad\textsuperscript{a}\thanks{CONTACT Sk. Safique Ahmad Email: safique@iiti.ac.in; Pinki Khatun Email: pinki996.pk@gmail.com} and Pinki Khatun\textsuperscript{a}}
\affil{\textsuperscript{a}Department  of Mathematics, Indian Institute of Technology Indore, Khandwa Road, Indore, 453552, Madhya Pradesh, India}
}

 \maketitle

\begin{abstract}
Perturbation theory plays a crucial role in sensitivity analysis, which is extensively used to assess the robustness of numerical techniques. 
		To quantify the relative sensitivity of any problem, it becomes essential to investigate structured condition numbers (CNs) via componentwise perturbation theory. This paper addresses and analyzes structured mixed condition number (MCN) and componentwise condition number (CCN) for the Moore-Penrose (M-P) inverse and the minimum norm least squares (MNLS) solution involving rank-structured matrices, which include the Cauchy-Vandermonde (CV) matrices and $\{1,1\}$-quasiseparable (QS) matrices. A general framework has been developed to compute the upper bounds for MCN and CCN of rank deficient parameterized matrices. This framework leads to faster computation of upper bounds of structured  CNs for CV and $\{1,1\}$-QS matrices.
		Furthermore, comparisons of obtained upper bounds are investigated theoretically and experimentally.
		In addition, the structured effective CNs for the M-P inverse and the MNLS solution of $\{1,1\}$-QS matrices are presented. Numerical tests reveal the reliability of the proposed upper bounds as well as demonstrate that the structured effective  CNs are computationally less expensive and can be substantially smaller compared to the unstructured CNs.
\end{abstract}


\section{Introduction}
Condition numbers (CNs) are one of the most important tools in matrix computation since they are widely employed to investigate the stability and robustness of numerical algorithms; see \cite{higham2002accuracy}. 
It measures how sensitive, in the worst-case scenario, a problem is to a slight change in input data and is crucial in determining whether a numerical solution makes sense.  On the other hand, the backward error is used to find a nearly perturbed problem with minimal magnitude perturbations so that the calculated solution becomes an actual solution to the perturbed problem. One can estimate the forward error of an approximate solution by combining the backward error with the CN (see \cite{burgisser2013condition} for more on  CNs).

Most likely, for the first time, Rice \cite{rice1966theory} introduced the classical theory of  CNs. In accordance with \cite{rice1966theory}, the normwise CN, which has been extensively considered in the literature, quantifies the input and output data errors using the norms.
The normwise CN has the drawback of ignoring the input and output data structures when the data is poorly scaled or sparse. Consequently, even minor relative normwise perturbations can have a disproportionate impact on entries that are small or zero, thereby potentially compromising data sparsity.
To address this challenge, mixed condition number (MCN) and componentwise condition number (CCN) \cite{skeel1979scaling, rohn1989new} emerge as a viable solution. 
MCN measures the input perturbations componentwise and the output error using norms, whereas CCN measures both the error in output data and the input perturbations componentwisely.

 The Moore-Penrose (M-P) inverse holds a pivotal position in matrix computation, offering a generalization of the standard inverse for rectangular or rank deficient matrices \(M\in \mathbb{R}^{m\times n}\). It is a unique matrix \(Y\in \mathbb{R}^{n\times m}\) that satisfies the following equations: \cite{stewart1990matrix}
\begin{align*}
	MYM=M ,\, YMY=Y, \,  (MY)^{\top}=MY, \, (YM)^{\top}=YM.
\end{align*}
Typically, it is denoted by $Y=M^{\dagger}.$ The M-P inverse finds its practical significance in solving the linear least squares (LS) problem, where the goal is to minimize the expression:
\begin{align}\label{eq11}
   \min_{z\in \R^{n}}\|Mz-b\|_2,
\end{align}
 where $M\in \R^{m\times n}$ and $b\in \R^m.$ Here, $\|\cdot\|_2$ denotes the spectral norm. If $M$ is of full column rank, the LS problem has a unique solution given by   $\x=M^{\dagger}b.$ However, when $M$ is rank deficient, the LS problem has infinitely many solutions, and the set of all solutions forms a closed and convex set. Consequently, this set uniquely yields the minimum norm LS (MNLS) solution $\x$, defined as $\x=M^{\dagger}b$. The M-P inverse and the LS problem have various  applications in digital image restoration and reconstruction \cite{imagerecon,imagerest}, Gauss–Markov
model \cite{Rao}, and so on. The literature on CNs for the M-P inverse \cite{NCN2007} and LS problems \cite{thesis}  is quite rich. The normwise CN for the M-P inverse and the LS problem is investigated in \cite{Gratton1996, Malyshev2003ls, NCN2007}, while MCN and CCN are considered in  \cite{cucker2007mixed,cucker2007mixed2}. For structured matrices, structured condition numbers for the M-P inverse and the LS problem have been investigated in \cite{xu2006condition, cucker2007mixed2}, which involves the preservation of the inherent matrix structure within the perturbation matrices.

In the past few years, many fast algorithms have been developed for various problems involving rank-structured matrices, such as computing eigenvalues and singular values \cite{vandebril2007matrix2,yang2021accurate}, solving linear systems \cite{vandebril2007matrix1} and LS problems \cite{huang2019accurate}, and computing the  M-P inverse \cite{chandrasekaran2002fast}. Quasiseparable (QS) \cite{eidelman1999new}, Cauchy \cite{huang2019accurate}, and Cauchy-Vandermonde (CV) \cite{huang2019accurate} matrices are popular examples of rank-structured matrices that arise in many applications such as in boundary value problem \cite{ greengard1991numerical, lee1997fast}, acoustic and electromagnetic scattering theory \cite{colton1998inverse}, interpolation problems \cite{muhlbach2000interpolation},
rational models of regression and E-optimal design \cite{imhof2001optimal}, and so on.

One of the striking properties of the rank-structured matrices is that they can be parameterized by $\mathcal{O}(m+n)$ parameters rather than $mn$ entries. Based on this property, many fast algorithms with lower computational costs have been developed \cite{vandebril2007matrix1, vandebril2007matrix2, vandebril2005note}. Plenty of works involving rank-structured matrices have been done in recent years to investigate the structured CNs for eigenvalue problems \cite{dopico2016structured, diao2019structured}, the solution of a linear system having a single as well as multiple right-hand sides \cite{martinez2016structured,meng2020structured}, the Sylvester matrix equation \cite{diao2023structured}, and so on, by considering perturbations on the parameters. Based on the above discussions, it is more sensible to investigate structured CNs by addressing perturbations on the parameters rather than directly on the matrix entries and to identify which set of parameters will be more suited for the development of fast algorithms. Thus, the forgoing discussion motivates us to consider perturbations on parameters instead of directly on entries in this paper. 

{In \cite{thesis}, authors have presented general parameterized QS representation and Givens-vector (GV) representation for the rectangular $m\times n ~(m\geq n)$ $\{1,1\}$-QS matrices (a special case of QS matrices), which are natural extensions of the square matrix case discussed in \cite{dopico2016structured, martinez2016structured, vandebril2005note}. Then, the authors studied the structured MCN for the LS problems when the coefficient matrix is a full column rank $m\times n$ $\{1,1\}$-QS matrices. Explicit expressions are derived with respect to both QS and GV representations. Additionally, comparison results between obtained structured CNs with respect to both representation and  with the upper bounds of the unstructured MCN are presented. Numerical experiments in \cite{thesis} showed that the structured MCNs can be much smaller than the unstructured ones. 
However, the above investigations do not address the rank deficient case. Furthermore, the MCN and CCN for the M-P inverse of rank deficient rank-structured matrices still need to be explored in the literature.} 
   Nevertheless, when dealing with rank deficient matrices, a prominent challenge in analyzing the CNs arises from the fact that even slight changes to the matrix can yield enormous variations in the computed M-P inverse. In light of this, normwise CNs for rank deficient unstructured matrices have been considered in \cite{wei2003condition,wei2007condition}, and for structured matrices in \cite{xu2006condition} under the assumptions: $\mathcal{R}(\D M)\subset \mathcal{R}(M)$ and $\mathcal{R}(\D M^{\top})\subset \mathcal{R}(M^{\top}),$ on the perturbation matrix $\D M$ in $M,$ where $\mathcal{R}(M)$ denotes the range of $M.$  Whereas, in \cite{yimin2005componentwise}, upper bounds are investigated for CCN for unstructured matrices under the above assumptions.  

This paper's central aim is to study the structured  MCN and CCN for the M-P inverse and the LS problem when dealing with rank deficient rank-structured matrices.
This investigation adheres to the rank-preserving constraint, denoted as $\text{rank}(M+\Delta M)=\text{rank}(M)$, which encompasses a broader class of perturbation matrices than those constrained by $\mathcal{R}(\Delta M)\subset \mathcal{R}(M)$ and $\mathcal{R}(\Delta M^{\top})\subset \mathcal{R}(M^{\top})$. This perspective expands the horizons of our study and offers valuable insights into structured CNs for this class of matrices.

The following highlights the main contributions of this paper: 
\begin{itemize}
	\item The MCN and CCN for two problems, the M-P inverse and the MNLS solution of the LS problem, involving rank deficient CV  and $\{1,1\}$-QS matrices are considered under the broader rank condition, i.e., $\text{rank}(M+\Delta M)=\text{rank}(M)$. 
	\item By considering matrix entries to be differentiable functions of a set of real parameters, we develop a general framework to compute the upper bounds of the MCN and CCN of the M-P inverse and LS problem for rank deficient parameterized matrices. In addition, exact expressions in the full column rank case of the MCN and CCN are also obtained.
	\item For the CV and $\{1,1\}$-QS matrices, compact upper bounds are obtained for structured MCN and CCN. Two important parameter representations for  $\{1,1\}$-QS matrices are considered: the  QS representation and GV representation.
	\item For $\{1,1\}$-QS matrices, structured effective CNs are proposed and shown that they can reliably estimate the actual conditioning of these matrices. Numerical experiments are reported to demonstrate that structured CNs are significantly smaller compared to unstructured CNs and align consistently with the theoretical results.
\end{itemize}

The remaining part of this paper is structured as follows. Section $\ref{s2}$ provides a few notations and preliminary results. In Section $\ref{s3},$ for the M-P inverse and the MNLS solution, we develop expressions of upper bounds for MCN and CCN for a general class of parameterized matrices. These frameworks are utilized in Sections $\ref{sec4}$ and $\ref{sec5}$ to derive the bounds for structured MCN and CCN for CV and $\{1,1\}$-QS matrices. Further, Section $\ref{sec5}$ studies comparison results between different structured and unstructured CNs. In Section $\ref{s6}$, numerical experiments are performed to illustrate our findings. Section $\ref{s7}$ ends with conclusions and a line of future research.
\section{Notation and preliminaries}\label{s2}  
The following notations are adopted throughout this work. $\R^{m\times n}$  denotes the collection of all $m\times n$ matrices  and $\R^m$ denotes the space of all column vectors of dimension $m,$ where $m,n$ are positive integers. To indicate the M-P inverse and  transpose of any matrix $M\in \mathbb{R}^{m\times n},$ $M^{\dagger}$ and $M^{\top}$ are used,  respectively. The $i$-th column of the identity matrix $I_m$ of order $m$ is denoted as $e_i^m$. For $M\in \R^{m\times n},$ set $\E_M:=I_m-MM^{\dagger}$ and $\F_M:=I_n-M^{\dagger}M.$ Denote ${\bf 0}$ as the zero matrix with conformal dimension. The Hadamard product of $M=[m_{ij}]\in \R^{m\times n}$ and $N=[n_{ij}]\in \R^{m\times n}$ is defined by $M\odot N: =[m_{ij} n_{ij}]\in \R^{m\times n}.$ For  $x=[x_i]\in \R^n$ and $M\in \R^{m\times n},$ the infinity and $\max$ norm are defined by $\|x\|_{\infty}:=\underset{i}{\text{max}} \,|x_i|$ and $ \|M\|_{\max}:=\underset{i,j}{\text{max}} \,|m_{ij}|,$ respectively. We denote $E_{ij}^{mn}=e_i^m(e_j^n)^{\top}$ as the matrix with $ij$-element is $1$ and zero elsewhere. The notation $i=1:n$ indicates that $i$ takes the values  $1,2,\ldots,n.$ For any $M\in \R^{m\times n},$ let $|M|=[|m_{ij}|].$  For matrices $M, N\in \R^{m\times n},$ $|M|\leq |N|$ means $|m_{ij}|\leq |n_{ij}|$ for $i=1 : m$ and $j=1:n.$ We define $M/N$ as  $(M/N)_{ij}=n_{ij}^{\ddagger}m_{ij},$ where for any $a\in \R,$ $a^{\ddagger}=\frac{1}{a}$ when $a\neq 0,$ otherwise  $a^{\ddagger}=1.$ ${\Theta}_{x}=\diag(x)\in \R^{n\times n}$ is the diagonal matrix, where $x\in \R^n.$ For any $a\in \R,$ $\sign(a):=\frac{a}{|a|}$  for $a\neq 0$ and $\sign(a):=0$ for $a=0$, and $\sign(M):=[\sign (m_{ij})].$ The $i$-th row and the $j$-th column of $M,$ are represented as $M(i,:)$ and $M(:,j),$ respectively.

Next, we discuss some important properties of the M-P inverse, which will be crucial for our main finding results. 
The following lemma states that for a full column rank matrices $M,$ its M-P inverse is a continuous function of its data entries.
\begin{lemma}\label{lm21}\cite{wang2018generalized}
	Let  $M\in \mathbb{R}^{m\times n}$ with full column rank and $\{E_j\}$ be a collection of real $m \times n$ matrices satisfying ${\displaystyle \lim_{j\rightarrow 0} }E_j={\bf 0}.$ Then, $(M+E_j)$ has full column rank when $j$ is small enough and ${\displaystyle \lim _{ j \rightarrow 0}}(M+E_j)^{\dagger}=M^{\dagger}.$
\end{lemma}
However, $M$ does not share the above property when it is singular or rank deficient. Small perturbation $\D M$ on $M$ can produce the computed M-P inverse far from the actual one. To tackle this situation, perturbation theory for the M-P inverse has been studied in certain specific constraints. Next, we recall the definition of `acute' perturbation.
\begin{definition}\cite{ stewart1990matrix}
     An acute perturbation $\tilde{M}=M+\D M\in \R^{m\times n}$ of $M\in \R^{m\times n}$ is a perturbed matrix  for which $\|MM^{\dagger}-\tilde{M}\tilde{M}^{\dagger}\|_2< 1$ and $\|M^{\dagger}M-\tilde{M}^{\dagger}\tilde{M}\|_2< 1,$ where $\|\cdot\|_2$ is the spectral norm.
\end{definition}
Proposition \ref{prop21} provides an if and only if condition for the continuity of $M^{\dagger}$ of any matrix $M\in \R^{m\times n}.$
\begin{proposition}\label{prop21}
     Let $M\in \R^{m\times n}.$ Consider the set $\S^1(M)=\big\{\D M\in \R^{m\times n}: \|M^{\dagger}\|_2\| \D M\|_2< 1\big\}.$ Then ${\displaystyle \lim_{\D M \rightarrow {\bf 0}}}(M+\D M)^{\dagger}=M^{\dagger}$ if and only if $\rank(M+\D M)=\rank(M),$ where $\D M\in \S^1(M).$
\end{proposition}
\renewcommand\qedsymbol{$\blacksquare$}
\begin{proof}
	For $\D M\in \S^1(M),$  we have $\|M^{\dagger}\|_2\| \D M\|_2< 1.$ Then $M+\D M$ is an acute perturbation of $M$ if and only if $\rank(M+\D M)=\rank(M)$ \cite[Lemma 1]{li2013note}. Since on the set of acute perturbations of $M,$ its  M-P inverse $M^{\dagger}$ is a continuous function about $M$ \cite[Page~$140$]{ stewart1990matrix}. Therefore, it follows that  $M^{\dagger}$ is continuous on the set $S^1(M)$ if and only if $\rank(M+\D M)=\rank(M).$ Hence, the proof is completed.
\end{proof}
\begin{remark}\label{re21}
	When $M$ has full column rank (or row rank),  from Lemma \ref{lm21}, the rank condition 
	in Proposition \ref{prop21} holds trivially.
\end{remark}
\section{MCN and CCN for  general parameterized  matrices}\label{s3}
In this part, initially, we define structured MCN and CCN for the M-P inverse and the unique MNLS solution for a general class of parameterized matrices. Suppose that each entry of $M\in \R^{m\times n}$ is a differentiable function of a set of real parameters $\Psi=[\psi_1,\psi_2,\ldots, \psi_p]^{\top}\in \mathbb{R}^p$ and write the matrix as $M(\Psi)$. We employ this notation for the rest of the paper. Due to the fact that a number of important classes of matrices can be parameterized by a collection of parameters, it is reasonable to consider perturbations on the parameters rather than directly on their entries. Let $\D \Psi\in \R^p$ be the perturbation on the parameters set $\Psi\in \R^p,$ we consider the  admissible perturbation in the matrix $M(\Psi)$ as $M(\Psi +\D \Psi)-M(\Psi)=\D M(\Psi)$.  For maintaining the continuity property for $M^{\dagger}(\Psi),$ according to the Proposition \ref{prop21}, we restrict the perturbation on $\Psi$ to the following set 
\begin{align}\label{eq31}
	\mathcal{S}(\Psi):=\Bigl\{\Delta \Psi \in \mathbb{R}^{p}:  \, \rank(M(\Psi))&=\rank(M(\Psi+\Delta \Psi)) =r, \, \|M^{\dagger}(\Psi)\|_2\| \D M(\Psi)\|_2 < 1\Bigr\}.
\end{align}
Next, we provide few examples to show that $\S(\Psi)$ is nonempty.
\begin{example}
    In this example, we consider a symmetric matrix $M(\Psi)=[m_{ij}]\in \R^{3\times 3},$ where $\Psi=\left[m_{11}=1, m_{21}=2, m_{31}=3, m_{22}=4, m_{32}=6, m_{33}=1\right]^{\top} \in \R^{6}.$ Then \begin{equation*}
        M(\Psi)=\bmatrix{1& 2& 3\\ 2 &4 &6\\ 3& 6 &1}.
    \end{equation*} 
    If we take $\Delta \Psi=\left[\Delta m_{11}=0, \Delta m_{21}=0, \Delta m_{31}=\mu, \Delta m_{22}=0,\Delta  m_{32}=2\mu, \Delta m_{33}=0\right]^{\top} \in \R^{6},$ where $\mu\neq -8/3$. Then, we get \begin{equation*}
          M(\Psi+\Delta \Psi )=\bmatrix{1& 2& 3+\mu\\ 2 &4 &6+2\mu\\ 3+\mu& 6+2\mu &1}.
    \end{equation*} 
    Here, $\|M^{\dagger}(\Psi)\|_2= 0.25$ 
 and $\| \Delta M^{\dagger}(\Psi)\|_2=\sqrt{3}|\mu|.$ Clearly, $\rank(M(\Psi))=\rank(M(\Psi+\Delta \Psi))$ and $\|M^{\dagger}(\Psi)\|_2\|\Delta M(\Psi)\|_2< 1, $ whenever $|\mu|< 2.3094.$ Therefore, the set $\S(\Psi)$ includes all perturbations $\Delta \Psi$ such that $|\mu|< 2.3094$.
\end{example}
\begin{example}
  Consider the parameter set $\Psi=[\{2,4\}, \{1\}, \{-3,1\}, \{5,2,6\}, \{1,3\}, \{2\}, \{3,1\}]^{\top} \in \R^{13}$ of a $\{1,1\}$-QS matrix given as in \eqref{eq51} and  using the formula provided in Definition \ref{def51}, we have  \begin{equation}
      M(\Psi)=\bmatrix{5 & 3 & 2 \\ -6& 2& 3\\ -12 & 4 &6}.
  \end{equation}
  Taking 
      $\Delta \Psi=[\{0,0\}, \{0\}, \{0,0\}, \{\mu,0,0\}, \{\mu,0\}, \{0\}, \{0,0\}]^{\top} \in \R^{13},$ 
 we get  \begin{equation*}
      M(\Psi +\Delta \Psi)=\bmatrix{5+\mu & 3+3\mu & 2+2\mu \\ -6& 2& 3\\ -12 & 4 &6}.
  \end{equation*}
  Here, $\|M^{\dagger}(\Psi)\|_2=0.1812$ 
 and $\| \Delta M^{\dagger}(\Psi)\|_2=\sqrt{14}|\mu|. $ Clearly, $\rank(M(\Psi))=\rank(M(\Psi+\Delta \Psi))$ and $\|M^{\dagger}(\Psi)\|_2\|\Delta M(\Psi)\|_2< 1, $ whenever $|\mu|< 1.4747.$ Thus, $\S(\Psi)$ contains all perturbations $\Delta \Psi$ such that $|\mu|< 1.4747.$
\end{example}
\begin{example}
    Consider the parameter set $\Psi=[\{1,2,3\}, \{4,5\}]^{\top} \in \R^{5}$ of a CV matrix and using the formula provided in Definition \ref{def:CV}, we have  \begin{equation*}
      M(\Psi)=\bmatrix{-\frac{1}{3} & -\frac{1}{4}& 1& 1 \\ -\frac{1}{2}&-\frac{1}{3}& 1&2\\ -1 & -\frac{1}{2}&1 &6}.
  \end{equation*}
  Considering $\Delta \Psi=[\{0,10^{-5},0\}, \{0,0\}]^{\top} \in \R^{5},$ we have 
  \begin{equation*}
      M(\Psi+\Delta \Psi)=\bmatrix{-\frac{1}{3} & -\frac{1}{4}& 1& 1 \\[0.5ex] -\frac{1}{2-10^{-5}}&-\frac{1}{3-10^{-5}}& 1&2+10^{-5}\\[0.5ex] -1 & -\frac{1}{2}&1 &6}.
  \end{equation*}
  Here, $\rank(M(\Psi))=\rank(M(\Psi+\Delta \Psi))=3$ and  we obtain $\|M^{\dagger}(\Psi)\|_2=7.5690
$ and $\|\Delta M(\Psi)\|_2= 1.0367\times 10^{-5}.
$ Consequently, $\|M^{\dagger}(\Psi)\|_2\|\Delta M(\Psi)\|_2=  7.8472\times 10^{-5}< 1, $ which shows that $\Delta \Psi\in \S(\Psi).$
\end{example}

\subsection{M-P inverse of  general parameterized matrices }
We introduce structured MCN and CCN in Definition \ref{def31} for M-P inverse of a general class of parameterized matrices. We provide general expressions for the upper bounds of these CNs in Theorem \ref{Th31}. Also, we present exact formulae for these CNs in Theorem \ref{Th32} for full
column rank matrices.

\begin{definition}\label{def31}
	Let $M(\Psi)\in \mathbb{R}^{m\times n},$  $\rank(M(\Psi))=r\leq \min{\{m,n\}}$. Then, we define structured MCN and CCN for $M^{\dagger}(\Psi)$ as follows:
	\begin{eqnarray*}
\mathscr{M}^{\dagger}\big(M(\Psi)\big)&:=&\lim_{\ep\rightarrow 0}\sup\left\{\frac{\|{M^{\dagger}(\Psi+\Delta \Psi)-M^{\dagger}(\Psi)}\|_{\max}}{\ep \|M^{\dagger}\|_{\max}} :\normx{\Delta \Psi /\Psi}_{\infty}\leq \ep, \Delta \Psi\in \mathcal{S}(\Psi)\right\},\\
\mathscr{C}^{\dagger}\big(M(\Psi)\big)&:=&\lim_{\ep\rightarrow 0}\sup\left\{\frac{1}{\ep}\normx{\frac{M^{\dagger}(\Psi+\Delta \Psi)-M^{\dagger}(\Psi)}{ M^{\dagger}}}_{\max}:\normx{\Delta \Psi /\Psi}_{\infty}\leq \ep, \Delta \Psi \in \mathcal{S}(\Psi)\right\}.
	\end{eqnarray*}
\end{definition}
To formulate the general expressions for the upper bounds of the CNs outlined in Definition $\ref{def31}$, we present the following perturbation expression for $M^{\dagger}(\Psi)$.

\begin{lemma}\label{lm31}
	Let $M(\Psi)\in \mathbb{R}^{m\times n}$ and $\rank(M(\Psi))=r$. Suppose $\Delta \Psi \in \mathcal{S}(\Psi)$ is the perturbation on the parameter set $\Psi.$ Then 
	\begin{align*}
		M^{\dagger}(\Psi +\Delta \Psi)-M^{\dagger}(\Psi)
		& =\sum _{k=1}^{p}\Bigl(-M^{\dagger}\frac{\partial M(\Psi)}{\partial \psi_k}M^{\dagger}+M^{\dagger}{M^{\dagger}}^{\top}\Bigl(\frac{\partial M(\Psi)}{\partial \psi_k}\Bigr)^{\top}\E_M\\
		&+\F_M\Bigl(\frac{\partial M(\Psi)}{\partial \psi_k}\Bigr)^{\top} {M^{\dagger}}^{\top}M^{\dagger}\Bigr)\Delta \psi_k +\mathcal{O}(\|\D\Psi\|_{\infty}^2).
	\end{align*}
	
\end{lemma}
\renewcommand\qedsymbol{$\blacksquare$}
\begin{proof}
	Given that the elements of $M(\Psi)$ are differentiable functions of $\Psi =[\psi_1, \psi_2,...,\psi_p]^{\top},$
	using matrix differentiation, for an infinitesimal change in $M(\Psi),$ we get
	\begin{equation}\label{Eq32}
		\Delta M(\Psi)=M(\Psi +\D \Psi)-M(\Psi)=\sum_{k=1}^{p}\frac{\partial M(\Psi)}{\partial \psi_k}\Delta \psi_k + \, \mathcal{O}(\|\Delta \Psi\|_{\infty}^2),
	\end{equation}
	where $\D \Psi=[\D\psi_1,\ldots, \D\psi_p]^{\top}.$
	Since $\D \Psi\in \S(\Psi),$ using the perturbation expression  for the M-P inverse \cite{stewart1990matrix}, we obtain
	\begin{align}\label{eqn33}
		\nonumber	\Delta M^{\dagger}(\Psi)=M^{\dagger}(\Psi+\Delta \Psi)-&M^{\dagger}(\Psi)=-M^{\dagger}\Delta M(\Psi)M^{\dagger}+M^{\dagger}{M^{\dagger}}^{\top}(\D M(\Psi))^{\top}\E_M\\
		&+\F_M (\D M(\Psi))^{\top} {M^{\dagger}}^{\top}M^{\dagger}+\, \mathcal{O}(\|\D \Psi\|_{\infty}^2).
	\end{align}
	Putting $(\ref{Eq32})$ in $(\ref{eqn33}),$ we get
	\begin{align*}
		\Delta M^{\dagger}(\Psi)=-M^{\dagger}&\Big(\sum_{k=1}^{p}\frac{\pr M(\Psi)}{\pr \psi_k}\D \psi_k \Big) M^{\dagger}+  M^{\dagger}{M^{\dagger}}^{\top}\Big(\sum_{k=1}^{p}\frac{\pr M(\Psi)}{\pr \psi_k}\D \psi_k\Big)^{\top} \E_M\\
		&+\F_M \Big(\sum_{k=1}^{p}\frac{\pr M(\Psi)}{\pr \psi_k}\D \psi_k\Big)^{\top} {M^{\dagger}}^{\top} M^{\dagger} +\, \mathcal{O}(\|\D \Psi\|_{\infty}^2).
	\end{align*}
	Hence, the desired expression is obtained. 
\end{proof}

In Theorem $\ref{Th31}$, for $M^{\dagger}(\Psi)$, we derive general expressions for upper bounds of the proposed CNs when $\rank(M(\Psi))=r.$
\begin{theorem}\label{Th31}
	For $M(\Psi)\in \mathbb{R}^{m\times n}$ with $\rank(M(\Psi))=r,$ we have
	\begin{align*}
		\M^{\dagger}\big(M(\Psi)\big)\leq \frac{ \| \mathcal{X}^{\dagger}_{\Psi}\|_{\max}}{\|M^{\dagger}\|_{\max}}  := \tilde{\M}^{\dagger}\big(M(\Psi)\big)\quad \mbox{and} \quad
		\C^{\dagger}\big(M(\Psi)\big)\leq  \Big \|\frac{\mathcal{X}^{\dagger}_{\Psi}}{M^{\dagger}} \Big\|_{\max}
		:= \tilde{\C}^{\dagger}\big(M(\Psi)\big),
	\end{align*}
 where \begin{align*}
       \mathcal{X}^{\dagger}_{\Psi}=\sum_{k=1}^{p}\Big( \Bigl \vert M^{\dagger} \frac{\pr M(\Psi)}{\pr \psi_k}M^{\dagger}\Bigr \vert + \vr M^{\dagger}{M^{\dagger}} ^{\top}\Bigl(\frac{\pr M(\Psi)}{\pr \psi_k}\Bigr)^{\top}\E _M\vr 
		+ \Bigl \vert\F_M \Bigl(\frac{\pr M(\Psi)}{\pr \psi_k}\Bigr)^{\top} {M^{\dagger}}^{\top} M^{\dagger}\Bigr\vert \Big) \vert \psi_k\vert.
 \end{align*}
\end{theorem}
\renewcommand\qedsymbol{$\blacksquare$}
 \begin{proof} 
From Lemma \ref{lm31} and using the properties of absolute values, we have
	\begin{align*}
		\Bigl \vert M^{\dagger}(\Psi+\Delta \Psi)-M^{\dagger}(\Psi)\Bigr \vert  \nonumber\leq &\sum _{k=1}^{p}\Big(\Bigl \vert M^{\dagger}\frac{\partial M(\Psi)}{\partial \psi_k}M^{\dagger}\Bigl \vert +\Bigl \vert M^{\dagger}{M^{\dagger}}^{\top}\Bigl(\frac{\partial M(\Psi)}{\partial \psi_k}\Bigr)^{\top}\E_M \Bigr \vert  \\
		&+\Bigl \vert \F_M \Bigl(\frac{\pr M(\Psi)}{\pr \psi_k}\Bigr) ^{\top} {M^{\dagger}}^{\top}M^{\dagger}\vr \Big) \vert \Delta \psi_k| \, +\mathcal{O}(\|\D \Psi\|_{\infty}^2).
	\end{align*}
	Now, by Definition \ref{def31}, $\Vert \Delta \Psi/\Psi \Vert_{\infty}\leq \ep $ implies that $\vert \D \psi_k\vert \leq \ep \, |\psi_k|$ for all $k=1,2,\ldots, p,$ and using the properties of the $\max$ norm, we find that
	\begin{equation*}
		\begin{split}
			\|M^{\dagger}(\Psi+\Delta \Psi)-M^{\dagger}(\Psi)\|_{\max}
			\leq &\, \ep\,  \Bigl\Vert \sum _{k=1}^{p}\Big( \vr M^{\dagger}\frac{\partial M(\Psi)}{\partial \psi_k}M^{\dagger}\Bigl \vert +\Bigl \vert M^{\dagger}{M^{\dagger}}^{\top}\Bigl(\frac{\partial M(\Psi)}{\partial \psi_k}\Bigr)^{\top}\E _M\Bigr \vert \\
			&+ \Bigl \vert \F_M \Bigl(\frac{\pr M(\Psi)}{\pr \psi_k}\Bigr) ^{\top} {M^{\dagger}}^{\top}M^{\dagger} \vr \Big) | \psi_k|\Bigr \Vert_{\max}+\mathcal{O}(\ep^2).
		\end{split}
	\end{equation*} 
	Then, if we take $\ep \rightarrow 0$, and from Definition \ref{def31}, we get the desired result of the first claim.\\
	In a similar manner, we obtain the second part of the claim. 
\end{proof}

In the next corollary, we obtain bounds for the CNs for $M^{\dagger}$ in the unstructured case.
\begin{corollary}\label{coro31}
	Suppose $M\in \R^{m\times n}$ with $\rank(M)=r.$ Then 
	\begin{eqnarray*}
		\tilde{\M}^{\dagger}(M)&=&\frac{1}{\|M^{\dagger}\|_{\max}}\Vr \vert M^{\dagger} \vert \vert M \vert \vert M^{\dagger}\vert + |M^{\dagger}{M^{\dagger}} ^{\top}| |M^{\top}| |\E_M|+|\F_M| |M^{\top}| |{M^{\dagger}} ^{\top}M^{\dagger}|\Vr_{\max},\\
		\tilde{\C}^{\dagger}(M)&=&\Vr \frac{1}{M^{\dagger}}\Bigl(\vert M^{\dagger} \vert \vert M \vert \vert M^{\dagger}\vert + |M^{\dagger}{M^{\dagger}} ^{\top}| |M^{\top}| |\E_M |+|\F_M | |M^{\top}| |{M^{\dagger}} ^{\top}M^{\dagger}|\Bigr)\Vr_{\max}.
	\end{eqnarray*}
\end{corollary}
\renewcommand\qedsymbol{$\blacksquare$}
\begin{proof}
	For any $M{=[m_{ij}]}\in \R^{m\times n}$ and any two column vectors $a$ and $b,$ we have
	\begin{eqnarray}\label{eqn34}
		\frac{\pr M}{\pr m_{ij}}&=&e_i^{m}(e_j^n)^{\top} \quad \text{and} \\ \label{Eq35}
	    |ab^{\top}|&=&|a||b^{\top}|.
	\end{eqnarray}
 	By considering the parameters as the entries of $M$ itself, i.e., $\Psi=[\{m_{ij}\}_{i,j=1}^{m,n}]^{\top}\in \R^{mn},$ and using $(\ref{eqn34}),$  the sum expression in Theorem \ref{Th31} can be written  as:		
	\begin{align}\label{eqn35}
		\nonumber	&\sum_{i=1}^{m}\sum_{j=1}^n\Big( \Bigl \vert M^{\dagger} \frac{\pr M}{\pr m_{ij}}M^{\dagger}\Bigr \vert + \vr M^{\dagger}{M^{\dagger}} ^{\top}\Bigl(\frac{\pr M}{\pr m_{ij}}\Bigr)^{\top}\E_M\vr  
		+\Bigl \vert\F_M \Bigl(\frac{\pr M}{\pr m_{ij}}\Bigr)^{\top} {M^{\dagger}}^{\top} M^{\dagger}\Bigr\vert \Big) \vert m_{ij}\vert \\ 
		&=\sum_{i=1}^{m}\sum_{j=1}^n\Bigl( \Bigl \vert M^{\dagger} e_i^{m}(e_j^{n})^{\top}M^{\dagger}\Bigr \vert + \vr M^{\dagger}{M^{\dagger}} ^{\top} e_j^{n}(e_i^{m})^{\top}  \E_M \vert  
		+\Bigl \vert \F_M  e_j^{n}(e_i^{m})^{\top} {M^{\dagger}}^{\top} M^{\dagger}\Bigr\vert \Bigr) \vert m_{ij} \vert.
	\end{align}
	Again, using $(\ref{Eq35}),$ we can write $(\ref{eqn35})$ as
	\begin{align}\label{eqn36}
		\nonumber& \sum_{i=1}^{m}\sum_{j=1}^n\Big( \vert M^{\dagger}(:,i) | \vert m_{ij} \vert |M^{\dagger}(j,:)\vert  + \vert M^{\dagger}{M^{\dagger}} ^{\top}(:,j) | \vert m_{ij} \vert |\E_M(i,:)\vert 
		\\ \nonumber
		&\quad \quad+ \vert\F_M(:,j)| \vert m_{ij} \vert | {M^{\dagger}}^{\top} M^{\dagger}(i,:)\vert  \Big) \\
		&=\vert M^{\dagger} \vert \vert M \vert \vert M^{\dagger}\vert + |M^{\dagger}{M^{\dagger}} ^{\top}| |M^{\top}| |\E_M |+|\F_M| |M^{\top}| |{M^{\dagger}} ^{\top}M^{\dagger}|.
	\end{align}
	The desired upper bounds will be obtained by substituting $(\ref{eqn36})$ in Theorem \ref{Th31}. 
\end{proof}

Next, we estimate the bounds for CNs under the constraints $\mathcal{R}(\D M(\Psi))\subseteq\mathcal{R}(M(\Psi))$ and $\mathcal{R}(\D M^{\top}(\Psi))\subseteq\mathcal{R}(M^{\top}(\Psi)).$
\begin{proposition}\label{pro31}
	Let $M(\Psi)\in \R^{m\times n}$ be such that $\rank(M(\Psi))=r.$ Suppose that $\D \Psi\in \R^{p}$ is the perturbation on the parameter set $\Psi$ satisfying the conditions, $\|M^{\dagger}\|_2\|\D M(\Psi)\|_2<1,$ $\mathcal{R}(\D M(\Psi))\subseteq\mathcal{R}(M(\Psi))$ and $\mathcal{R}(\D M^{\top}(\Psi))\subseteq\mathcal{R}(M^{\top}(\Psi)).$ Then
	\begin{eqnarray*}
		M^{\dagger}(\Psi +\Delta \Psi)-M^{\dagger}(\Psi)
		 =-\sum _{k=1}^{p}M^{\dagger}\frac{\partial M(\Psi)}{\partial \psi_k}M^{\dagger} \D \psi_k+\mathcal{O}(\|\D\Psi\|_{\infty}^2).
	\end{eqnarray*}
	Furthermore, \begin{eqnarray*}
	\tilde{\M}^{\dagger}\big(M(\Psi)\big)&=&\frac{\|\sum _{k=1}^{p}|M^{\dagger}\frac{\partial M(\Psi)}{\partial \psi_k}M^{\dagger} | |\D \psi_k|\|_{\max}}{\|M^{\dagger}\|_{\max}},\\
	\tilde{\C}^{\dagger}\big(M(\Psi)\big)&=&\Vr\frac{1}{M^{\dagger}}\sum _{k=1}^{p}|M^{\dagger}\frac{\partial M(\Psi)}{\partial \psi_k}M^{\dagger} | |\D \psi_k|\Vr_{\max}.
	\end{eqnarray*}
\end{proposition}
\renewcommand\qedsymbol{$\blacksquare$}
\begin{proof}
	If $M(\Psi),\D M(\Psi)\in \R^{m\times n}$ satisfies the assumptions  $\mathcal{R}(\D M(\Psi))\subseteq\mathcal{R}(M(\Psi))$ and $\mathcal{R}(\D M^{\top}(\Psi))\subseteq\mathcal{R}(M^{\top}(\Psi)).$ Then 
	\begin{equation}\label{eqn37}
		M(\Psi)M^{\dagger}(\Psi)\D M(\Psi)=\D M(\Psi), \quad M^{\top}(\Psi){M^{\dagger}(\Psi)}^{\top}\D M^{\top}(\Psi)=\D M^{\top}(\Psi).
	\end{equation}
	In addition, if $\|M^{\dagger}(\Psi)\|_2 \|\D M(\Psi)\|_2 < 1$ holds, it is shown in \cite{ben1966error} that
	\begin{equation}\label{eqn38}
		M^{\dagger}(\Psi+\D\Psi)=(I_n+M^{\dagger}(\Psi)\D M(\Psi))^{-1}M^{\dagger}(\Psi).
	\end{equation}
	Now, $(\ref{eqn37})$ implies $\D M^{\top}(\Psi)\E_M={\bf 0}$ and $\F_M\D M^{\top}(\Psi)={\bf 0}.$ Again, $(\ref{eqn38})$ implies that $\rank(M(\Psi+\D\Psi))=\rank(M(\Psi)).$
	Therefore, $\D \Psi\in \S(\Psi).$ Hence, from Lemma \ref{lm31}, we get the desired expression. 
	
 The proof of the second part is a direct consequence of the derived perturbation expansion and the proof technique employed in Theorem $\ref{Th31}.$ 
\end{proof}
\begin{remark}\label{ref31}
	Using Proposition \ref{pro31}, and in an analogous approach to Corollary \ref{coro31}, we can recover the bounds for unstructured CNs obtained in \cite{yimin2005componentwise}.
\end{remark}
For the matrices with full column rank, the next theorem provides exact expressions of CNs for $M^{\dagger}(\Psi),$ introduced in Definition \ref{def31}.
\begin{theorem}\label{Th32}	
	For $M(\Psi)\in \mathbb{R}^{m\times n}$ with full column rank, we have
	\begin{equation*}
		\M^{\dagger}(M(\Psi))=\frac{\|\Hat{\mathcal{X}}^{\dagger}_{\Psi}\|_{\max}}{\Vert M^{\dagger}\Vert_{\max}}\quad \mbox{and} \quad
		\C^{\dagger}(M(\Psi))= \left\Vert\frac{\Hat{\mathcal{X}}^{\dagger}_{\Psi}}{ M^{\dagger}}\right\Vert_{\max},
	\end{equation*}
 where \begin{eqnarray*}
     \Hat{\mathcal{X}}^{\dagger}_{\Psi}= \sum _{k=1}^{p}\vr M^{\dagger}\frac{\partial M(\Psi)}{\partial \psi_k} M^{\dagger}-(M^{\top}M)^{-1}\Big(\frac{\partial M(\Psi)}{\partial \psi_k}\Big)^{\top}\E_M \vr | \psi_k|.
 \end{eqnarray*}
\end{theorem}
\renewcommand\qedsymbol{$\blacksquare$}
\begin{proof}
	Since $M(\Psi)$ is of full column rank matrix,  we have $\F_M={\bf 0}.$ Now, applying Remark \ref{re21}  on Lemma $\ref{lm31}$, we get following perturbation expression for $M^{\dagger}(\Psi)$
	\begin{align}\label{eqn39}
		\nonumber    \D M^{\dagger}(\Psi)=M^{\dagger}(\Psi +\D \Psi)-M^{\dagger}(\Psi)&=\sum_{k=1}^{p} \Big(-M^{\dagger}\frac{\partial M(\Psi)}{\partial \psi_k} M^{\dagger}+(M^{\top}M)^{-1}\\
		&\Big(\frac{\partial M(\Psi)}{\partial \psi_k}\Big)^{\top}\E_M\Big)\D \psi_k+\mathcal{O}(\|\D\Psi\|_{\infty}^2).
	\end{align}
	By employing a similar proof technique as in Theorem $\ref{Th31}$, and considering the given condition  $\Vert \Delta \Psi/\Psi \Vert_{\infty}\leq \ep,$ we can establish the following bound:
	\begin{equation}\label{eqn310}
		\M^{\dagger}(M(\Psi))\leq \frac{\Bigl \Vert \sum _{k=1}^{p}\vr M^{\dagger}\frac{\partial M(\Psi)}{\partial \psi_k} M^{\dagger}-(M^{\top}M)^{-1}\Big(\frac{\partial M(\Psi)}{\partial \psi_k}\Big)^{\top}\E_M \vr | \psi_k|\Bigr\Vert_{\max}}{\Vert M^{\dagger}\Vert_{\max}}.
	\end{equation}
	On the other hand, from Lemma \ref{lm21} and Remark \ref{re21}, it follows that we can choose arbitrary perturbation $\D \Psi\in \R^{p}$ on the parameter set $\Psi.$ Choose 
	\begin{equation}\label{eqn311}
		\Delta \psi_k=- \ep \, \sign(M_k )_{lq}\, \sign(\psi_k)\, \psi_k,
	\end{equation} 
	where $M_k=M^{\dagger}\frac{\partial M(\Psi)}{\partial \psi_k} M^{\dagger}-(M^{\top}M)^{-1}\Big(\frac{\partial M(\Psi)}{\partial \psi_k}\Big)^{\top}\E_M,$ for $k=1:p,$  $(M_k)_{lq}$ denotes the $lq$-th entry of the matrix $M_k,$ and  the indices $l$ and $q$ are such that
	\begin{equation*}
		\Bigl\Vert \sum _{k=1}^{p}\vert M_k \vert | \psi_k|\Bigr\Vert_{\max}=\Bigl( \sum _{k=1}^{p}\vert M_k\vert | \psi_k|\Bigr)_{lq}.
	\end{equation*}
	The upper bound in $(\ref{eqn310})$ is obtained by inserting the values of $(\ref{eqn311})$ in the perturbation expression $(\ref{eqn39})$ and from the Definition \ref{def31}. Therefore, the proof of the first claim is concluded.\\
	The second part of the claim can be obtained in a similar approach. 
\end{proof} 

\subsection{MNLS solution for general parameterized coefficient matrices}
Let us consider the LS problem \eqref{eq11} for the parameterized matrix $M(\Psi)\in \R^{m\times n}$
\begin{equation} \label{eqn312}
	\min_{z\in \R^{n}}\|M(\Psi)z-b\|_2
\end{equation}
  with $\rank(M(\Psi))=r$ and $b\in \R^m.$ When $M(\Psi)$ is rank deficient, the unique MNLS solution is provided by $\x=M(\Psi)^{\dagger}b.$ Moreover, in this situation, $\x$ is not even a continuous function of the data, and small changes in $M(\Psi)$ can produce large changes to $\x.$ This happens as a consequence of the behavior of the M-P inverse for any rank deficient matrix. Thus, according to Proposition \ref{prop21}, we consider the perturbation $\D \Psi\in \S(\Psi)$ for the parameters, and then the perturbed problem $$\min_{z \in \R^{n}}\|M(\Psi +\D \Psi)z-(b+\D b)\|_2$$  has the MNLS solution $\tilde{\x}=M(\Psi +\D \Psi)^{\dagger}(b+\D b).$ Consider $\D \x=\tilde{\x}-\x$.

In Definition \ref{def32}, for the MNLS solution $\x,$ we introduce its structured MCN and CCN.
\begin{definition}\label{def32}
	Let $M(\Psi)\in \mathbb{R}^{m\times n}$ with $\rank(M(\Psi))=r$ and $b\in \mathbb{R}^{m}$. Then, we define structured MCN and CCN of $\x$  as follows:
	\begin{eqnarray*}
		\M^{\dagger}(M(\Psi),b)&:=&\lim_{\ep \rightarrow 0}\sup \left\{ \frac{\|\Delta \x\|_{\infty}}{\ep \|\x\|_{\infty}}:\normx{\Delta \Psi/\Psi}_{\infty}\leq \ep, \normx{\Delta b/b}_{\infty}\leq \ep ,\, \D\Psi\in \S(\Psi) , \D b\in \mathbb{R}^m \right\},\\
		\C^{\dagger}(M(\Psi),b)&:=&\lim_{\ep \rightarrow 0}\sup \left\{\frac{1}{\ep}\left \lVert \frac{\Delta \x} {\x}\right \rVert_{\infty}:\normx{\Delta \Psi/\Psi}_{\infty}\leq \ep, \normx{\Delta b/b}_{\infty}\leq \ep,\,  \D\Psi \in\S(\Psi), \D b\in \mathbb{R}^m \right\}.
	\end{eqnarray*}
\end{definition}
Our main objective of this section is to find general expressions of bounds for the CNs  introduced in Definition \ref{def32}, and the following lemma provides the perturbation expansion for the MNLS solution.
\begin{lemma}\label{lm32}
	Let $M(\Psi)\in \mathbb{R}^{m\times n}$ with $\rank(M(\Psi))=r$ and $b\in \mathbb{R}^{m}$. Suppose $\Delta \Psi \in \S(\Psi)$ and $\Delta b\in \R^m,$ 
	and set ${\bf r}:=b-M(\Psi)\x.$ Then
	\begin{align*}
		\Delta \x =\sum _{k=1}^{p}&\left(-M^{\dagger}\frac{\partial M(\Psi)}{\partial \psi_k} \x + M^{\dagger}{M^{\dagger}}^{\top}\Bigl(\frac{\partial M(\Psi)}{\partial \psi_k}\Bigr)^{\top} \r+\F_M \Big(\frac{\pr M(\Psi)}{\pr \psi_k}\Big)^{\top}{M^{\dagger}}^{\top} \x \right)\Delta \psi_k\\ &+\sum_{i=1}^mM^{\dagger} e_i^{m} \Delta b_i +\mathcal{O}(\|[\D \Psi,\D b]\|_{\infty}^2).
	\end{align*} 
\end{lemma}
\renewcommand\qedsymbol{$\blacksquare$}
\begin{proof}
	Since $\Delta \x= M^{\dagger}(\Psi +\Delta \Psi)(b+\Delta b)-M^{\dagger}(\Psi)b$ and $\Delta \Psi \in \S(\Psi),$ using Lemma \ref{lm31}, we get 
	\begin{align}\label{eqn313}
		\Delta \x \nonumber&= \Bigg(\sum _{k=1}^{p}\Bigl(-M^{\dagger}\frac{\partial M(\Psi)}{\partial \psi_k}M^{\dagger}+M^{\dagger}{M^{\dagger}}^{\top}\Bigl(\frac{\partial M(\Psi)}{\partial \psi_k}\Bigr)^{\top}\E_M\\
		&\quad+\F_M\Bigl(\frac{\partial M(\Psi)}{\partial \psi_k}\Bigr)^{\top} {M^{\dagger}}^{\top}M^{\dagger}\Bigr)\Delta \psi_k +M^{\dagger}(\Psi)\Bigg)(b+\Delta b)-M^{\dagger}(\Psi) b+\mathcal{O}(\|\D \Psi\|_{\infty}^2)\\ \nonumber
		&=\sum _{k=1}^{p}\Bigl(-M^{\dagger}\frac{\partial M(\Psi)}{\partial \psi_k}\x+M^{\dagger}{M^{\dagger}}^{\top}\Bigl(\frac{\partial M(\Psi)}{\partial \psi_k}\Bigr)^{\top}\r+\F_M\Bigl(\frac{\partial M(\Psi)}{\partial \psi_k}\Bigr)^{\top} {M^{\dagger}}^{\top}\x\Bigr)\Delta \psi_k \\ \label{EQ315}
		& \quad \quad+M^{\dagger}(\Psi)\Delta b+\, \mathcal{O}(\|[\D \Psi,\D b]\|_{\infty}^2).
	\end{align}
	On the other hand, for the perturbation $\D b$ in $b,$ we can write $\Delta b=\sum_{i=1}^m e_i^m\Delta b_i.$ Therefore, from \eqref{EQ315},  we get 
	\begin{align*}
		\Delta \x= &\sum _{k=1}^{p}\Bigl(-M^{\dagger}\frac{\partial M(\Psi)}{\partial \psi_k}\x+M^{\dagger}{M^{\dagger}}^{\top}\Bigl(\frac{\partial M(\Psi)}{\partial \psi_k}\Bigr)^{\top}\r+\F_M\Bigl(\frac{\partial M(\Psi)}{\partial \psi_k}\Bigr)^{\top} {M^{\dagger}}^{\top}\x\Bigr)\Delta \psi_k \\
		&+M^{\dagger} \sum_{i=1}^m e_i^m\Delta b_i+ \, \mathcal{O}(\|[\D \Psi,\D b]\|_{\infty}^2),
	\end{align*}
	and hence, the desired result is obtained. 
\end{proof}

In Theorem \ref{Th33}, we provide general expressions for the upper bounds of $\M^{\dagger}(M(\Psi),b)$ and $ \C^{\dagger}(M(\Psi),b).$
\begin{theorem}\label{Th33}
	Let $M(\Psi) \in \mathbb{R}^{m\times n}$ be a matrix having $\rank(M(\Psi))= r$ and $b\in \mathbb{R}^m.$ Then, 
	\begin{eqnarray*}
		\M^{\dagger}(M(\Psi),b)&\leq& \frac{\|  \mathcal{X}^{ls}_{\Psi}\|_{\infty}}{\Vert \x \Vert_{\infty}}:=\tilde{\M}^{\dagger}(M(\Psi),b),\\
		\C^{\dagger}(M(\Psi),b) &\leq &\Big\Vert  \Theta_{\x^\ddagger} \mathcal{X}^{ls}_{\Psi} \Big\Vert_{\infty}:=\tilde{\C}^{\dagger}(M(\Psi),b),
	\end{eqnarray*}
	where 
 \begin{eqnarray*}
     \mathcal{X}^{ls}_{\Psi}= \sum _{k=1}^{p}\Big(\vr M^{\dagger}\frac{\partial M(\Psi)}{\partial \psi_k}\x\vr+\vr M^{\dagger}{M^{\dagger}}^{\top}\Bigl(\frac{\partial M(\Psi)}{\partial \psi_k}\Bigr)^{\top} \r \vr
		 +\vr \F_M \Big(\frac{\pr M(\Psi)}{\pr \psi_k}\Big)^{\top} {M^{\dagger}}^{\top} \x\vr\Big) |\psi_k| + |M^{\dagger}||b|
 \end{eqnarray*}
 and $\Theta_{\x^{\ddagger}}=\diag(\x^{\ddagger}).$
\end{theorem}
\renewcommand\qedsymbol{$\blacksquare$}
\begin{proof}  
From Lemma \ref{lm32} and utilizing the properties of absolute values, we obtain
	\begin{align}\label{eqn314}
		|\D \x|
		\nonumber\leq \sum _{k=1}^{p}\Bigg(&\vr M^{\dagger}\frac{\partial M(\Psi)}{\partial \psi_k} \x\vr+\vr M^{\dagger}  {M^{\dagger}}^{\top}\Bigl(\frac{\partial M(\Psi)}{\partial \psi_k}\Bigr)^ T \r \vr+ \vr \F_M \Big(\frac{\pr M(\Psi)}{\pr \psi_k}\Big)^{\top}{M^{\dagger}}^{\top} \x\vr\Bigg) \vert \Delta \psi_k\vert\\
		& +\sum_{i=1}^m \vert M^{\dagger} \vert \vert \D b_i\vert +  \mathcal{O}(\|[\D \Psi,\D b]_{\infty}\|^2).
	\end{align}
	Now, by Definition \ref{def32}, $\Vert \Delta \Psi/\Psi \Vert_{\infty}\leq \ep $ and $\Vert \D b/b\Vert_{\infty} \leq \ep $ implies that for $k=1:p,$ $\vert \D \psi_k\vert \leq \ep \, |\psi_k|$, and for  $i=1:m,$  $\vert \D b_i| \leq \ep |b_i|,$ respectively. Taking  infinity norm in $(\ref{eqn314})$, we deduced that
	\begin{align} \label{eqn315}
		\nonumber	\|\D \x\|_{\infty}
		&\leq \ep \, \Vr\sum _{k=1}^{p}\Big(\vr M^{\dagger}\frac{\partial M(\Psi)}{\partial \psi_k} \x\vr +\vr  M^{\dagger}  {M^{\dagger}}^{\top}\Bigl(\frac{\partial M(\Psi)}{\partial \psi_k}\Bigr)^ T \r\\
		&+\vr \F_M \Bigl(\frac{\pr M(\Psi)}{\pr \psi_k}\Bigr)^{\top}{M^{\dagger}}^{\top} \x\vr\Big) \vert \psi_k\vert + | M^{\dagger}| |b| \Vr_{\infty} +  \,  \mathcal{O}(\ep^2).
	\end{align} 
	Then, if we take $\ep \rightarrow 0$ in $(\ref{eqn315})$ and from Definition \ref{def32}, we attain the desired result of the first assertion.
The second assertion follows in a similar manner, as we can express $\left\|\frac{\D\x}{\x}\right\|_{\infty}=\left\|\Theta_{\x^{\ddagger}} \D \x\right\|_{\infty}.$ 
\end{proof}

Next, we discuss the bounds of the CNs for the MNLS solution of the LS problem $(\ref{eqn312})$ corresponding to unstructured matrices.
\begin{corollary}\label{coro34}
	For $M\in \R^{m\times n}$ having $\rank(M)=r$ and $b\in \R^m,$ we have
	\begin{eqnarray*}
		\tilde{\M}^{\dagger}(M,b)&:=&\frac{\Vr\vert M^{\dagger} \vert \vert M \vert \vert \x\vert + |M^{\dagger}{M^{\dagger}} ^{\top}| |M^{\top}| |\r|+|\F_M| |M^{\top}| |{M^{\dagger}} ^{\top}\x|+|M^{\dagger}||b|\Vr_{\infty}}{\|\x\|_{\infty}},\\
		\tilde{\C}^{\dagger}(M,b)&:=&\Vr \Theta_{\x^{\ddagger}}\Bigl(\vert M^{\dagger} \vert \vert M \vert \vert {\x}\vert + |M^{\dagger}{M^{\dagger}} ^{\top}| |M^{\top}| |\r|+|\F_M| |M^{\top}| |{M^{\dagger}} ^{\top} \x|+|M^{\dagger}||b| \Bigr)\Vr_{\infty}.
	\end{eqnarray*}
\end{corollary}
\renewcommand\qedsymbol{$\blacksquare$}
\begin{proof}
 The proof follows in an analogous way to the Corollary \ref{coro31} by considering $\Psi=[\{m_{ij}\}_{i,j=1}^{m,n}]^{\top}\in \R^{mn}$ in Theorem \ref{Th33}    and using $(\ref{eqn34})$ and \eqref{Eq35}. 
\end{proof}
The next theorem offers explicit formulae of structured CNs for the unique LS solution $\x=M^{\dagger}(\Psi)b$ for full column rank matrices.
\begin{theorem}\label{Th34}
	For $M(\Psi) \in \mathbb{R}^{m\times n}$ having full column rank and $b\in \mathbb{R}^m,$ we get
\begin{eqnarray*}
\M^{\dagger}\big(M(\Psi),b\big)=\frac{\|\Hat{\mathcal{X}}^{ls}_{\Psi} \|_{\infty}}{\Vert \x \Vert_{\infty}}\quad\mbox{and}\quad
\C^{\dagger}\big(M(\Psi),b\big)=\Vr\Theta_{\x^\ddagger}\Hat{\mathcal{X}}^{ls}_{\Psi} \Vr_{\infty},
	\end{eqnarray*}
	where \begin{eqnarray}
	   \Hat{\mathcal{X}}^{ls}_{\Psi}= \sum _{k=1}^{p}\vr M^{\dagger}\frac{\partial M(\Psi)}{\partial \psi_k}\x-(M^{\top}M)^{-1}\Bigl(\frac{\partial M(\Psi)}{\partial \psi_k}\Bigr)^{\top} \r \vr | \psi_k| \,+ \,|M^{\dagger}||b|.
	\end{eqnarray}
\end{theorem}
\renewcommand\qedsymbol{$\blacksquare$}
\begin{proof} 
	In Lemma \ref{lm32}, using the fact that for any full column rank matrix $\F_M={\bf 0},$ we have 
	\begin{equation}\label{eqn316}
		\Delta \x =\sum _{k=1}^{p}\left(-M^{\dagger}\frac{\partial M(\Psi)}{\partial \psi_k} \x + M^{\dagger}{M^{\dagger}}^{\top}\Bigl(\frac{\partial M(\Psi)}{\partial \psi_k}\Bigr)^{\top} \r \right)\Delta \psi_k +\sum_{i=1}^mM^{\dagger} e_i^{m} \Delta b_i +\mathcal{O}(\|[\D \psi,\D b]\|_{\infty}^2).
	\end{equation}
	Now, by applying the proof method used in Theorem $\ref{Th32}$ and considering the given conditions $\|\D \Psi/\Psi\|_{\infty}\leq \ep$ and $\|\D b/b\|_{\infty}\leq \ep,$ we obtain
	\begin{align}\label{eqn317}
		\M^{\dagger}(M(\Psi),b)\leq \frac{\Vr \sum _{k=1}^{p}\vr M^{\dagger}\frac{\partial M(\Psi)}{\partial \psi_k}\x-(M^{\top}M)^{-1}\Bigl(\frac{\partial M(\Psi)}{\partial \psi_k}\Bigr)^{\top} \r \vr | \psi_k| + |M^{\dagger}||b| \Vr_{\infty}}{\Vert \x \Vert_{\infty}}.
	\end{align}
	From Lemma \ref{lm21} and Remark \ref{re21}, we can choose the perturbation $\D \Psi$ on the parameters set $\Psi$ arbitrarily from $\R^p.$ Consider the following perturbations
	\begin{center}
		$\Delta b=\ep \, \mathbf{\Theta}_{M^{\dagger}}\mathbf{\Theta}_{b}\,b,$
	\end{center} 
	where $\mathbf{\Theta}_{M^{\dagger}}$  and $\mathbf{\Theta}_{b}$ are the  diagonal matrices having diagonal entries $\mathbf{\Theta}_{M^{\dagger}}(i,i)=\sign(M^{\dagger}(l,i)),$ and $\mathbf{\Theta}_b(i,i)=\sign(b_i)$ for $i=1:m,$ respectively, and
	\begin{center}
		$\Delta \psi_k=- \ep \, \sign(M_{\x,k})_l \,\sign(\psi_k)\, \psi_k,$
	\end{center} 
	where $M_{\x,k}:=M^{\dagger}\frac{\partial M(\Psi)}{\partial \psi_k}\x-(M^{\top}M)^{-1}\left(\frac{\partial M(\Psi)}{\partial \psi_k}\right)^{\top} \r$  and  $l$ is the index so that 
	\begin{equation*}
			\Bigl\Vert \sum _{k=1}^{p}\vert M_{\x,k} \vert | \psi_k| + |M^{\dagger}||b|\Bigr\Vert_{\infty}=\Big(\sum _{k=1}^{p}\vert M_{\x,k} \vert | \psi_k| + |M^{\dagger}||b|\Big)_l. 
	\end{equation*}
	The upper bound in $(\ref{eqn317})$ will be attained by substituting these perturbations in $(\ref{eqn316})$ and from Definition $\ref{def32},$ and hence the desired expression is attained.
	Analogously, we can obtain the expression for the CCN.
\end{proof}
{\begin{remark} The formula for the MCN $\M^{\dagger}(M(\phi),b)$ in Theorem \ref{Th34} is also presented in \cite{thesis}. However, our approach to proving the theorem differs slightly. Interested readers may also refer to the proof method in \cite{thesis}. 
\end{remark}}
\section{\label{sec4} Cauchy-Vandermonde (CV) matrices}
In this section, we start by reviewing the definition of CV matrices. Subsequently,  we provide the derivative expressions for the matrix with respect to its parameter set. These expressions play a pivotal role in the derivation of computationally feasible upper bounds for the structured  CNs of the M-P inverse and the solution of the LS problem for a rank deficient CV matrix given in Theorems \ref{Th41} and \ref{Th43}, respectively.  Explicit formulations for these  CNs are also provided in the Theorems \ref{Th42} and \ref{Th44}, respectively, when the matrix has full column rank.
\begin{definition} \label{def:CV}\cite{huang2019accurate}
A matrix $M\in \R^{m\times n}$ is classified as a CV matrix if it satisfies the following conditions: for  $c=[c_1,c_2,\ldots, c_m]^{\top}\in \mathbb{R}^m$ and $d=[d_1,d_2,\ldots, d_l]^{\top}\in \mathbb{R}^l$, where $c_i\neq d_j$ for $i=1:m$ and $j=1:l$, with $0\leq l \leq n$, the matrix $M$ can be represented as follows:
	\begin{equation}\label{eq41}
		\begin{split}
			M &= \bmatrix{
				\frac{1}{c_1-d_1} & \frac{1}{c_1-d_2} & \cdots & \frac{1}{c_1-d_l} & 1& c_1 & c_1^2 &\cdots & c_1^{n-l-1} \\
				\frac{1}{c_2-d_1} & \frac{1}{c_2-d_2} & \cdots & \frac{1}{c_2-d_l} & 1& c_2 & c_2^2 &\cdots & c_2^{n-l-1}\\
				\vdots & \vdots & \vdots & \vdots &\vdots &\vdots &\vdots & \vdots & \vdots\\
				\frac{1}{c_m-d_1} & \frac{1}{c_m-d_2} & \cdots & \frac{1}{c_m-d_l} & 1& c_m & c_m^2 &\cdots & c_m^{n-l-1}
			}.
		\end{split}
	\end{equation}
 $M$ becomes the Vandermonde matrix when $l=0,$ and the Cauchy matrix when $l=n.$ 
\end{definition} 
For a CV matrix of the form $(\ref{eq41}),$  $\Psi_{\mathbb{CV}}:=[\{c_i\}_{i=1}^m,\{d_i\}_{i=1}^l]^{\top}\in \mathbb{R}^{m+l}$ represents its parameter set. We use the notation $M(\Psi_{\mathbb{CV}})$ to refer a CV matrix parameterized by $\Psi_{\mathbb{CV}}$.

For the M-P inverse and the MNLS solution involving the CV matrix, our objective is to estimate the structured  CNs. Lemma \ref{lm41} accomplishes our claim. Before that, we will construct the following matrices. For any positive integers $p,q$ and any vector $y=[y_1,y_2,\ldots , y_p]^{\top}\in \mathbb{R}^p,$   define the matrices 
\begin{equation*}
	\mathcal{Q}_{y,i}^{pq}:=\bmatrix{{\bf 1}, \ldots, {\bf 1}, y,{\bf 1}, \ldots,{\bf 1}}
	\in \mathbb{R}^{p\times q},
\end{equation*}
for $i=1: q,$ with the $i$-th column is $y$ and ${\bf 1}\in \R^p$  have all entries equal to 1.
Also, set 
\begin{align}\label{eq42}
    \mathcal{M}_1&:= \bmatrix{-M(\Psi_{\mathbb{CV}})(:,1:l)& {\bf 0}& M(\Psi_{\mathbb{CV}})(:,l+2:n)}\in \R^{m\times n},\\ \label{eq43}
&\mathcal{M}_2:=\bmatrix{M(\Psi_{\mathbb{CV}})(:,1:l) & {\bf 0}}\in \R^{m\times n},
\end{align}
where ${\bf 0} $ is the zero matrix with conformal dimensional.

The following lemma presents the derivative expressions of a CV matrix for the parameters in $\Psi_{\mathbb{CV}}.$ 
\begin{lemma}\label{lm41}
	Suppose $M(\Psi_{\mathbb{CV}})\in \R^{m\times n}$ having rank $r,$ represented by a set of real parameter $\Psi_{\mathbb{CV}}=\left[\{c_i\}_{i=1}^m,\{d_i\}_{i=1}^l\right]^{\top}\in \mathbb{R}^{m+l},$ where $c_i\neq d_j,$ $i=1:m$ and $j=1 :l.$  Then, each entry of $M(\Psi_{\mathbb{CV}})$ is a differentiable function of $\Psi_{\mathbb{CV}},$ and
 \begin{enumerate}
     \item $\frac{\partial M(\Psi_{\mathbb{CV}})}{\partial c_i} = e_i^{m}(\mathcal{M}_1 \odot(\mathcal{Q}^{nm}_{{\mathbf c}'_i,i})^{\top} )(i,:)$  for $i=1:m,$\label{itm:1b}
     \item $\frac{\partial M(\Psi_{\mathbb{CV}})}{\partial d_j} = (\mathcal{M}_2\odot \mathcal{Q}^{mn}_{\mathbf{d}'_j,j})(:,j)(e_j^n)^{\top}$  for $j=1:l,$\label{itm:2b}
 \end{enumerate}
	where 
 \begin{align*}
     {\bf c}'_i&:=\Big[\frac{1}{c_i-d_1}, \frac{1}{c_i-d_2},\ldots, \frac{1}{c_i-d_l}, 1, \frac{1}{c_i}, \frac{2}{c_i},\ldots, \frac{(n-l-1)}{c_i}\Big]^{\top}\in \mathbb{R}^n,\\
     &{\bf d}'_j:=\Big[\frac{1}{c_1-d_j}, \frac{1}{c_2-d_j},\ldots,\frac{1}{c_m-d_j}\Big]^{\top}\in \mathbb{R}^m,
 \end{align*}
	for $i=1:m$ and $j=1:l.$
\end{lemma}
\renewcommand\qedsymbol{$\blacksquare$}
\begin{proof} By observing that, when $c_i\neq d_j,$ where $i=1:m$ and $j=1: l,$ partial derivatives corresponding to the parameters $\{c_i\}_{i=1}^m$ will be
	$$\frac{\pr M(\Psi_{\mathbb{CV}})}{\pr c_i}=\bmatrix{
		0 & \cdots & 0 & 0& 0 & 0 &\cdots & 0 \\
		\vdots &  & \vdots &\vdots &\vdots &\vdots & & \vdots \\
		0 & \cdots & 0 &0 & 0&0 &\cdots & 0 \\
		\frac{-1}{(c_i-d_1)^2} & \cdots & \frac{-1}{(c_i-d_l)^2}& 0 & 1& 2c_i &\cdots & (n-l-1)c_i^{n-l-2}\\ 
		0 & \cdots& 0 & 0 & 0& 0& \cdots & 0 \\   
		\vdots &  & \vdots &\vdots &\vdots& \vdots& & \vdots \\
		0 & \cdots & 0 &0 & 0 & 0& \cdots & 0}
	.$$
	Now, on the right-hand side of the above, using Hadamard product with the matrix $(\mathcal{Q}^{nm}_{\mathbf{c}'_i, i})^{\top},$ we get
	\begin{align*}
		\frac{\pr M(\Psi_{\mathbb{CV}})}{\pr c_i}&= \Big( e_i^m\bmatrix{-M(\Psi_{\mathbb{CV}})(i,1:l)& 0 &M(\Psi_{\mathbb{CV}})(i,l+2:n)}\Big)\odot (\mathcal{Q}^{nm}_{\mathbf{c}'_i,i})^{\top}\\
		&=\Big(e_i^m\mathcal{M}_1(i,:)\Big)\odot  (\mathcal{Q}^{nm}_{{\bf c}'_i,i})^{\top}=e_i^{m}(\mathcal{M}_1\odot (\mathcal{Q}^{nm}_{{\bf c}'_i,i})^{\top})(i,:).
	\end{align*}
	Hence, proof of the first statement follows.
	
	In a similar argument, we can prove the second part of the statement.
\end{proof}
For the structured  CNs,  computationally feasible upper bounds are provided in the following theorem for $M^{\dagger}(\Psi_{\mathbb{CV}})$ addressed in Definition \ref{def31}.
\begin{theorem}\label{Th41}
	Suppose $M(\Psi_{\mathbb{CV}})\in \mathbb{R}^{m\times n}$ with $\rank(M(\Psi_{\mathbb{CV}}))=r.$ Then
	\begin{eqnarray*}
		\tilde{\M}^{\dagger}(M(\Psi_{\mathbb{CV}}))=\frac{\|  \mathcal{X}^{\dagger}_{\mathbb{CV}}\|_{\max}}{\|M^{\dagger}\|_{\max}} \quad \mbox{and}\quad
		\tilde{\C}^{\dagger}(M(\Psi_{\mathbb{CV}}))=\left\| \frac{\mathcal{X}^{\dagger}_{\mathbb{CV}}}{M^{\dagger}}\right\|_{\max},
	\end{eqnarray*}
	where \begin{align*}
\mathcal{X}^{\dagger}_{\mathbb{CV}}=&|M^{\dagger}||\Theta_{c}||(\mathcal{M}_1\odot \mathcal{Q} )M^{\dagger}|+ |M^{\dagger}{M^{\dagger}}^{\top}(\mathcal{M}_1\odot \mathcal{Q} )^{\top}||\Theta_{c}||\mathbf{E}_M|
		\\
		&+|\mathbf{F}_M(\mathcal{M}_1\odot \mathcal{Q} )^{\top}||\Theta_{c}||{M^{\dagger}}^{\top}M^{\dagger}|+|M^{\dagger}(\mathcal{M}_2\odot \mathcal{M}_2)||\Theta_{d'}||M^{\dagger}|\\
		&+|M^{\dagger}{M^{\dagger}}^{\top}| |\Theta_{d'}||(\mathcal{M}_2\odot \mathcal{M}_2)^{\top}\mathbf{E}_M|+|\mathbf{F}_M||\Theta_{d'}||(\mathcal{M}_2\odot \mathcal{M}_2)^{\top}{M^{\dagger}}^{\top}M^{\dagger}|,
  \end{align*}
 $d'=[d_1,\ldots, d_l,0,\ldots,0]^{\top}\in \R^n,$ $\mathcal{Q}=[{{\bf c}'_1}^{\top},{{\bf c}'_2}^{\top},\ldots, {{\bf c}'}_m^{\top}]^{\top}\in\R^{m\times n}$ and $\mathcal{M}_1,\mathcal{M}_2$ as defined in $(\ref{eq42})$ and $(\ref{eq43}),$ respectively. 
\end{theorem}
\renewcommand\qedsymbol{$\blacksquare$}
\begin{proof}
	For deriving the desired expressions for $\tilde{\M}^{\dagger}(M(\Psi_{\mathbb{CV}}))$ and $\tilde{\C}^{\dagger}(M(\Psi_{\mathbb{CV}})),$  we calculate the contribution of each parameter subset to the expressions outlined in Theorem   \ref{Th31}. 
 For the parameters $\{c_i\}_{i=1}^l,$ using \eqref{itm:1b} of Lemma \ref{lm41} we get:
	\begin{align}\label{eq44}
		\nonumber	\mathcal{E}_c:&=\sum_{i=1}^{m}\Big(\vr M^{\dagger}\frac{\partial M(\Psi_{\mathbb{CV}})}{\partial c_i}M^{\dagger}\vr +\vr {M^{\dagger}}^{\top} M^{\dagger}\Bigl(\frac{\partial M(\Psi_{\mathbb{CV}})}{\partial c_i}\Bigr)^{\top} \E_M \Bigr\vert  \\ \nonumber
		& \quad +\vr \F_M\Big(\frac{\pr M(\Psi_{\mathbb{CV}})}{\pr c_i}\Big)^{\top}{M^{\dagger}}^{\top} M^{\dagger}\vr\Big) |c_i|\\ \nonumber
		&=\sum_{i=1}^m\Big(\vr M^{\dagger}e_i^{m}((\mathcal{Q}^{nm}_{{\bf c}'_i,i})^{\top}\odot \mathcal{M}_1)(i,:)M^{\dagger}\vr +\vr M^{\dagger}{M^{\dagger}}^{\top}\Big(e_i^{m}( (\mathcal{Q}^{nm}_{{\bf c}'_i,i})^{\top}\odot \mathcal{M}_1)(i,:)\Big)^{\top}\E_M \vr\\ 
		&\quad +\vr \F_M \Big(e_i^{m}( (\mathcal{Q}^{nm}_{{\bf c}'_i,i})^{\top}\odot \mathcal{M}_1)(i,:)\Big)^{\top}{M^{\dagger}}^{\top}M^{\dagger}\vr \Big)|c_i|.
	\end{align}
	Using $(\ref{Eq35})$ in $(\ref{eq44}),$ we get
	\begin{align*}
		\mathcal{E}_c&=\sum_{i=1}^m\Big(|M^{\dagger}(:,i)| |c_i| |( (\mathcal{Q}^{nm}_{{\bf c}'_i,i})^{\top}\odot \mathcal{M}_1)(i,:)M^{\dagger}| + |M^{\dagger}{M^{\dagger}}^{\top}(\mathcal{Q}^{nm}_{{\bf c}'_i,i})^{\top}\odot \mathcal{M}_1)^{\top})^{\top}(i,:)| |c_i||\E_M(i,:)|\\
		&\quad+|\F_M (\mathcal{Q}^{nm}_{{\bf c}'_i,i})^{\top}\odot \mathcal{M}_1)^{\top}(i,:)||c_i| |{M^{\dagger}}^{\top}M^{\dagger}(i,:)|\Big)\\
		&=|M^{\dagger}||\Theta_{c}||(\mathcal{M}_1\odot \mathcal{Q})M^{\dagger}|+|M^{\dagger}{M^{\dagger}}^{\top} (\mathcal{M}_1\odot \mathcal{Q})^{\top}||\Theta_{c}||\E_M|+|\F_M(\mathcal{M}_1\odot \mathcal{Q})^{\top}||\Theta_c||{M^{\dagger}}^{\top}M^{\dagger}|.
	\end{align*}
	Analogously, for the parameters $\{d_i\}_{i=1}^l,$ using \eqref{itm:2b} of Lemma \ref{lm41}, we have
	\begin{align*}
		\mathcal{E}_d&:=\sum_{i=1}^{l}\Big(\vr M^{\dagger}\frac{\partial M(\Psi_{\mathbb{CV}})}{\partial d_i}M^{\dagger}\vr +\vr {M^{\dagger}}^{\top} M^{\dagger}\Bigl(\frac{\partial M(\Psi_{\mathbb{CV}})}{\partial d_i}\Bigr)^{\top} \mathbf{E}_M \Bigr\vert\\
		& \quad+\vr \F_M \Big(\frac{\pr M(\Psi_{\mathbb{CV}})}{\pr d_i}\Big)^{\top}{M^{\dagger}}^{\top} M^{\dagger}\vr\Big) |d_i|\\
		&=|M^{\dagger}(\mathcal{M}_2\odot \mathcal{M}_2)||\Theta_{d'}||M^{\dagger}|+|M^{\dagger}{M^{\dagger}}^{\top}||\Theta_{d'}||(\mathcal{M}_2\odot \mathcal{M}_2)^{\top}\E_M|\\
		&\quad +|\F_M||\Theta_{d'}||(\mathcal{M}_2\odot \mathcal{M}_2)^{\top} {M^{\dagger}}^{\top}M^{\dagger}|.
	\end{align*}
	Applying Theorem $\ref{Th31}$ yields $$\tilde{\M}^{\dagger}(M(\Psi_{\mathbb{CV}}))=\frac{\|\mathcal{E}_d+\mathcal{E}_c\|_{\max}}{\|M^{\dagger}\|_{\max}} \quad \mbox{and}\quad \tilde{\C}^{\dagger}(M(\Psi_{\mathbb{CV}}))=\Vr \frac{\mathcal{E}_d+\mathcal{E}_c}{M^{\dagger}}\Vr_{\max}.$$
	Hence, the proof is completed.
\end{proof}
 We now employ Theorem \ref{Th32} to deduce the explicit formulations for structured CNs of $M^{\dagger}(\Psi_{\mathbb{CV}})$ by considering $M(\Psi)$ has full column rank, which is presented next.
\begin{theorem}\label{Th42}
	Suppose $M(\Psi_{\mathbb{CV}})\in \mathbb{R}^{m\times n}$ has full column rank. Then,
	\begin{eqnarray*}
		\M^{\dagger}(M(\Psi_{\mathbb{CV}})) =\frac{\Vert \Hat{\mathcal{X}}^{\dagger}_{\mathbb{CV}}\Vert_{\max}}{\|M^{\dagger}\|_{\max}}\quad \mbox{and}\quad
		\C^{\dagger}(M(\Psi_{\mathbb{CV}})) =
		\left\| \frac{\Hat{\mathcal{X}}^{\dagger}_{\mathbb{CV}}}{M^{\dagger}}\right\|_{\max},
	\end{eqnarray*}
	where 
 \begin{align*}
\Hat{\mathcal{X}}^{\dagger}_{\mathbb{CV}}=&\sum_{i=1}^{m}\Bigl\vert M^{\dagger}E_{ii}^{mm}(\mathcal{M}_1\odot \mathcal{Q})M^{\dagger} -(M^{\top}M)^{-1}(E_{ii}^{mm}(\mathcal{M}_1\odot \mathcal{Q}))^{\top} \E_M \Bigr\vert | c_i|\\
		&+\sum_{i=1}^{l}\Bigl\vert M^{\dagger}(\mathcal{M}_2\odot\mathcal{M}_2)E_{jj}^{nn}M^{\dagger} -(M^{\top}M)^{-1}((\mathcal{M}_2\odot\mathcal{M}_2)E_{jj}^{nn})^{\top} \E_M \Bigr\vert | d_i|,
 \end{align*}
 $E_{ij}^{mn}=e_i^m(e_j^{n})^{\top}$ and $\mathcal{Q}$ is as defined in Theorem \ref{Th41}.
\end{theorem}
\begin{proof} 
	To employ the expressions given in Theorem \ref{Th32}, we need to compute the contribution for each subset of parameters in an analogous method to the proof of Theorem \ref{Th41}.  For the parameters $\{c_i\}_{i=1}^m,$ using \eqref{itm:1b} of Lemma \ref{lm41}, we have
	\begin{align*}
		\mathcal{E}'_c&:=\sum_{i=1}^{m}\Bigl\vert M^{\dagger}\frac{\partial M(\Psi_{\mathbb{CV}})}{\partial c_i}M^{\dagger} -(M^{\top}M)^{-1}\Bigl(\frac{\partial M(\Psi_{\mathbb{CV}})}{\partial c_i}\Bigr)^{\top}\E_M \Bigr\vert | c_i|\\
		&=\sum_{i=1}^{m}\Bigl\vert M^{\dagger}\Bigl(e_i^m(\mathcal{M}_1\odot (\mathcal{Q}^{nm}_{{\bf c}'_i,i})^{\top})(i,:)\Bigr)M^{\dagger}-(M^{\top}M)^{-1}\Bigl(e_i^m(\mathcal{M}_1\odot (\mathcal{Q}^{nm}_{{\bf c}'_i,i})^{\top})(i,:)\Bigr)^{\top} \E_M \Bigr\vert | c_i|\\
		&=\sum_{i=1}^{m}\Bigl\vert M^{\dagger}E_{ii}^{mm}(\mathcal{M}_1\odot \mathcal{Q})M^{\dagger} -(M^{\top}M)^{-1}(E_{ii}^{mm}(\mathcal{M}_1\odot \mathcal{Q}))^{\top} \E_M \Bigr\vert | c_i|.
	\end{align*}
	Similarly, for the parameters $\{d_i\}_{i=1}^l, $ using \eqref{itm:2b} of Lemma \ref{lm41}, we get
	\begin{align*}		
		\mathcal{E}'_d:&=\sum_{i=1}^{l}\Bigl\vert M^{\dagger}\frac{\partial M(\Psi_{\mathbb{CV}})}{\partial d_i}M^{\dagger} -(M^{\top}M)^{-1}\Bigl(\frac{\partial M(\Psi_{\mathbb{CV}})}{\partial d_i}\Bigr)^{\top} \E_M \Bigr\vert | d_i|\\
		&=\sum_{i=1}^{l}\Bigl\vert M^{\dagger}(\mathcal{M}_2\odot\mathcal{M}_2)E_{jj}^{nn}M^{\dagger} -(M^{\top}M)^{-1}((\mathcal{M}_2\odot\mathcal{M}_2)E_{jj}^{nn})^{\top} \E_M \Bigr\vert | d_i|.
	\end{align*}
	From Theorem \ref{Th32}, we have
	$$ \M^{\dagger}(M(\Psi_{\mathbb{CV}}))=\frac{\Vert \mathcal{E}'_c+\mathcal{E}'_d \Vert_{\max}}{\Vert M^{\dagger}\Vert_{\max}} \quad \mbox{and}\quad  \C^{\dagger}(M(\Psi_{\mathbb{CV}}))=\Bigl\Vert \frac{ \mathcal{E}'_c+\mathcal{E}'_d}{M^{\dagger}} \Bigr\Vert_{\max},$$
	and hence, our proof is completed.
\end{proof}
\begin{remark}
    If we consider $l=0$ or $l=n$ in the preceding results, we can calculate the structured CNs for the M-P inverse Vandermonde and Cauchy matrices, respectively.
\end{remark}
Next, we consider the LS problem $(\ref{eqn312})$ corresponding to a rank deficient CV matrix.  Using the expressions given in  Theorem $\ref{Th33},$ we deduce upper bounds for structured CNs of $\x,$ presented next.
\begin{theorem}\label{Th43}
	Suppose $M(\Psi_{\mathbb{CV}})\in \mathbb{R}^{m\times n}$ with rank $r$  and $b\in \mathbb{R}^{m}.$ Set $\r:=b-M(\Psi_{\mathbb{CV}})\x$. Then,
\begin{eqnarray*}
\tilde{\M}^{\dagger}\big(M(\Psi_{\mathbb{CV}}),b\big)=&\frac{\|\mathcal{X}^{ls}_{\mathbb{CV}}\|_{\infty}}{\|\x\|_{\infty}}\quad \mbox{and}\quad \tilde{\C}^{\dagger}\big(M(\Psi_{\mathbb{CV}},b)\big)=\Vr \Theta_{\x^{\ddagger}} \mathcal{X}^{ls}_{\mathbb{CV}}\Vr_{\infty},
	\end{eqnarray*}
 where \begin{align*}
\mathcal{X}^{ls}_{\mathbb{CV}}=&|M^{\dagger}||b|+|M^{\dagger}||\Theta_c||(\mathcal{M}_1\odot \mathcal{Q})\x|+|M^{\dagger}{M^{\dagger}}^{\top}(\mathcal{M}_1\odot \mathcal{Q})^{\top}||\Theta_c||\r|\\
		&+|\F_M (\mathcal{M}_1\odot \mathcal{Q})^{\top}||\Theta_c||{M^{\dagger}}^{\top}\x|+|M^{\dagger}(\mathcal{M}_2\odot \mathcal{M}_2)||\Theta_{d'}||\x|\\
		&+|M^{\dagger}{M^{\dagger}}^{\top}||\Theta_{d'}||\mathcal{M}_2\odot \mathcal{M}_2)^{\top}\r|+|\F_M ||\Theta_{d'}||(\mathcal{M}_2\odot \mathcal{M}_2)^{\top} {M^{\dagger}}^{\top}\x|.
 \end{align*}
\end{theorem}
\begin{proof}
	By evaluating in an analogous method to the proof of Theorem \ref{Th41}, for each subset of parameters using the expressions given in Theorem \ref{Th33}, the proof is followed. Hence, we omit it here.
\end{proof}
The structured CNs to the LS problem $(\ref{eqn312})$  corresponding to a full column rank CV matrix are stated next.
\begin{theorem}\label{Th44}
	Let $M(\Psi_{\mathbb{CV}})\in \mathbb{R}^{m\times n }$ and $b\in \mathbb{R}^{m}.$ Set $\r:=b-M\x$. Then,
	\begin{eqnarray*}
		\M^{\dagger}(M(\Psi_{\mathbb{CV}}),b) =\frac{\Vert\Hat{\mathcal{X}}^{ls}_{\mathbb{CV}}\Vert_{\infty}}{\|\x\|_{\infty}} \quad \mbox{and} \quad
		\C^{\dagger}(M(\Psi_{\mathbb{CV}}),b) =\Bigl\Vert \Theta_{\x^{\ddagger}}\Hat{\mathcal{X}}^{ls}_{\mathbb{CV}}\Bigr\Vert_{\infty},
	\end{eqnarray*}
	where \begin{align*}
\Hat{\mathcal{X}}^{ls}_{\mathbb{CV}}=&\sum_{i=1}^{m}\Bigl\vert M^{\dagger}E_{ii}^{mm}(\mathcal{M}_1\odot \mathcal{Q})\x -(M^{\top}M)^{-1}(E_{ii}^{mm}(\mathcal{M}_1\odot \mathcal{Q}))^{\top} \r \Bigr\vert | c_i|\\
		&+\sum_{i=1}^{l}\Bigl\vert M^{\dagger}(\mathcal{M}_2\odot\mathcal{M}_2)E_{jj}^{nn}\x -(M^{\top}M)^{-1}((\mathcal{M}_2\odot\mathcal{M}_2)E_{jj}^{nn})^{\top} \r+|M^{\dagger}||b|.
	\end{align*}
\end{theorem}
\begin{proof}
	Since the proof follows in an analogous method to the proof of Theorem \ref{Th42}, by finding the contribution of each parameter set in the expressions of Theorem \ref{Th34}. Hence, we omit the proof. 
\end{proof}
\section{Quasiseparable (QS) matrices}\label{sec5}
The outset of this section begins with a quick introduction to QS matrices, which is a specific type of rank-structured matrices. Specifically, CNs are investigated for two important representations known as QS representation \cite{dopico2016structured, eidelman1999new} and GV representation \cite{vandebril2005note}. Upper bounds for the CNs of the M-P inverse and the MNLS solution are obtained corresponding to QS representation in Subsection \ref{ss51} and for the GV representation in Subsection  \ref{ss52}. The relationship between different CNs is also investigated in Subsection \ref{ss53}.

For the first time, QS matrices were investigated in \cite{eidelman1999new}. In this work, we focus solely on considering $\{1,1\}$-QS matrices, which is a special case of QS matrices.  Let $M$ be in $\R^{n\times n}.$ If every submatrix of $M$ completely contained in the strictly lower triangular (resp. upper triangular) part is of rank at most $1$ (resp. $1$), and there is at least one of these submatrices has rank equal to $1$ (resp. $1$), then $M$ is called a $\{1,1\}$-QS matrix. Equivalently, we can write: ${\dm{\max_{i}}\, \rank}(M(i+1:n,1:i))=1$ and ${\dm{\max_{i}}\, \rank}(M(1:i,i+1:n))=1.$ 

\subsection{CNs corresponding to QS representation}\label{ss51}
In \cite{dopico2016structured, eidelman1999new}, the notion of QS representation was proposed for $\{1,1\}$-QS matrices. In this subsection, we first recall this representation and then discuss the structured MCN and CCN.
\begin{definition}\label{def51}\cite{dopico2016structured, eidelman1999new}
	A matrix $M\in \mathbb{R}^{n\times n}$ is classified to be a $\{1,1\}$-QS matrix if it can be parameterized by the following set of $7n-8$ real parameters
	\begin{equation}\label{eq51}
		\Psi_{\mathbb{QS}}=\Bigl[\{{\bf a_i}\}_{i=2}^n, \{{\bf e_i}\}_{i=2}^{n-1},\{{\bf b_i}\}_{i=1}^{n-1}, \{{\bf d_i}\}_{i=1}^n, \{\f_i\}_{i=1}^{n-1}, \{\g_i\}_{i=2}^{n-1}, \{{\bf{h_i}}\}_{i=2}^n\Bigr]^{\top}\in \R^{7n-8},
	\end{equation}
	as follows,
	\begin{equation*}\label{eq52}
		M=\bmatrix{
			{\bf d} _1& \f_1\h_2 & \f_1\g_2\h_3& \cdots & \f_1\g_2\cdots \g_{n-1}\h_n\\
			\a_2 \b_1 & {\bf d}_2& \f_2\h_3 & \cdots & \f_2\g_3\cdots \g_{n-1}\h_n\\
			\a_3\e_2\b_1 & \a_3\b_2 & {\bf d}_3 & \cdots & \f_3\g_4\cdots \g_{n-1}\h_n\\
			\a_4\e_3\e_2\b_1 & \a_4\e_3\b_2 & \a_4\b_3 & \cdots & \f_4\g_5\cdots \g_{n-1}\h_n\\
			\vdots & \vdots & \vdots & \ddots & \vdots\\
			\a_n\e_{n-1}\ldots \e_2\b_1 & \a_n\e_{n-1}\ldots \e_3\b_2 & \a_n\e_{n-1}\ldots \e_4\b_3 & \cdots & {\bf d}_n\\
		}.
	\end{equation*}
\end{definition}
The set of real parameters $\Psi_{\mathbb{QS}}$ as in $(\ref{eq52})$ is called  \emph{QS representation} of $M.$ We use the notation $M(\Psi_{\mathbb{QS}})$ to refer a $\{1,1\}$-QS matrix parameterized by the set $\Psi_{\mathbb{QS}}.$ For the rest part, we set 	\begin{equation}\label{Eq52}
	M(\Psi_{\mathbb{QS}}):=\mathrm{L}_M+\mathrm{D}_M+\mathrm{U}_M,
\end{equation} where $\L_M$ and $\U_M$ denote  the strictly lower and upper triangular part of $M(\Psi_{\mathbb{QS}}),$ respectively, and $\Da_M$ denotes the diagonal part of $M(\Psi_{\mathbb{QS}}).$ 

In Lemma \ref{lm51}, we recall the derivative expressions of a $\{1,1\}$-QS matrix $M(\Psi_{\mathbb{QS}})$ for the parameters in $\Psi_{\mathbb{QS}}$ provided in Definition \ref{def51}, which are useful to obtain the desired upper bounds for CNs. These results are discussed in \cite{dopico2016structured}.
\begin{lemma}\label{lm51}\cite{dopico2016structured}
	Let $M(\Psi_{\mathbb{QS}})\in \mathbb{R}^{n\times n}$ be a $\{1,1\}$-QS matrix. Then $M(\Psi_{\mathbb{QS}})$ has entries that are differentiable functions of $\Psi_{\mathbb{QS}}$ defined as in \eqref{eq51} and
\begin{enumerate}
    \item $\frac{\pr M(\Psi_{\mathbb{QS}})}{\pr {\bf d}_i}=e_i^{n}(e_i^{n})^{\top},$ for $i=1:n$.\label{itm:1}
    \item  $\a_i\frac{\pr M(\Psi_{\mathbb{QS}})}{\pr \a_i}=e_i^{n}\L_M(i,:),$ for $i=2:n$.\label{itm:2}
    \item $\e_i \frac{\pr M(\Psi_{\mathbb{QS}})}{\pr \e_i}=\bmatrix{
			{\bf 0} & {\bf 0}\\
			M(\Psi_{\mathbb{QS}})(i+1:n,1:i-1) & {\bf 0}
		}:=\mathcal{F}_i, $ for $i=2:n-1.$\label{itm:3}
  \item $\b_i\frac{\pr M(\Psi_{\mathbb{QS}})}{\pr \b_i}=\L_M(:,i)(e_i^{n})^{\top},$ for $i=1:n-1$.\label{itm:4}
  \item $\g_i \frac{\pr M(\Psi_{\mathbb{QS}})}{\pr \g_i}=\bmatrix{
			{\bf 0} & M(\Psi_{\mathbb{QS}})(1:i-1,i+1:n) \\
			{\bf 0} & {\bf 0}
		}:=\mathcal{G}_i,$ for  $i=2:n-1.$\label{itm:5}
  \item $\f_i\frac{\pr M(\Psi_{\mathbb{QS}})}{\pr \f_i}=e_i^{n} \U_M(i,:),$ for $i=1:n-1$.\label{itm:6}
  \item  $\h_i \frac{\pr M(\Psi_{\mathbb{QS}})}{\pr \h_i}=\U_M(i,:)(e_i^{n})^{\top},$ for $i=2:n$.\label{itm:7}
\end{enumerate}
\end{lemma}

Next, we use the derivative expressions given  in Lemma \ref{lm51} and Theorem \ref{Th31} to compute the bounds of the structured CNs for $M^{\dagger}(\Psi_{\mathbb{QS}}).$
\begin{theorem}\label{Th51}
	For $M(\Psi_{\mathbb{QS}})\in \mathbb{R}^{n\times n}$ with $\rank(M(\Psi_{\mathbb{QS}}))=r,$ we have
	\begin{eqnarray*}
		\tilde{\M}^{\dagger}(M(\Psi_{\mathbb{QS}}))=\frac{\| \mathcal{X}^{\dagger}_{\mathbb{QS}}\|_{\max}}{\|M^{\dagger}\|_{\max}}\quad and \quad \tilde{\C}^{\dagger}(M(\Psi_{\mathbb{QS}}))=\Vr\frac{\mathcal{X}^{\dagger}_{\mathbb{QS}}}{M^{\dagger}}\Vr_{\max},
	\end{eqnarray*}
	where 
  \vspace{-1mm}
  \begin{align*}
\mathcal{X}^{\dagger}_{\mathbb{QS}}:=&|M^{\dagger}||\Da_M||M^{\dagger}|+|M^{\dagger}{M^{\dagger}}^{\top}||\Da_M||\E_M|+|\F_M||\Da_M||{M^{\dagger}}^{\top}M^{\dagger}|+|M^{\dagger}||\L_MM^{\dagger}|\\
		&+|M^{\dagger}{M^{\dagger}}^{\top}\L_M^{T}|| \E_M|+|\F_M \L_M^{\top}||{M^{\dagger}}^{\top}M^{\dagger}|+|M^{\dagger}\L_M||M^{\dagger}|+|M^{\dagger}{M^{\dagger}}^{\top}||\L_M^{\top}\E_M|\\
		&+|\F_M||\L_M^{\top}{M^{\dagger}}^{\top}M^{\dagger}|+|M^{\dagger}||\U_M M^{\dagger} |+|M^{\dagger}{M^{\dagger}}^{\top}\U_M^{\top}||\E_M|
		+|\F_M \U_M^{\top}||{M^{\dagger}}^{\top}M^{\dagger}|\\
		&+|M^{\dagger}\U_M||M^{\dagger}|+|M^{\dagger}{M^{\dagger}}^{\top}||\U_M^{\top}\E_M|+|\F_M|| \U_M^{\top}{M^{\dagger}}^{\top}M^{\dagger}|\\
		&+\sum_{i=2}^{n-1}\Big(|M^{\dagger}\mathcal{F}_iM^{\dagger}|+|M^{\dagger}{M^{\dagger}}^{\top}\mathcal{F}_i^{\top}\E_M|+|\F_M\mathcal{F}_i^{\top}{M^{\dagger}}^{\top}M^{\dagger}|\Big)\\
		&+\sum_{j=2}^{n-1}\Big(|M^{\dagger}\mathcal{G}_jM^{\dagger}|+|M^{\dagger}{M^{\dagger}}^{\top}\mathcal{G}_j^{\top}\E_M|+|\F_M \mathcal{G}_j^{\top}{M^{\dagger}}^{\top}M^{\dagger}|\Big),
	\end{align*}
		 $ \mathcal{F}_i$ and $\mathcal{G}_i$ are defined as in Lemma \ref{lm51}.
\end{theorem}
\begin{proof} The proof of the above assertions involves determining the contribution of each subset of parameters to the expressions provided in Theorem \ref{Th31}. Using \eqref{itm:1} of Lemma \ref{lm51} for the parameters $\{{\bf d}_i\}_{i=1}^n$,  we have
	\begin{align}\label{eq53}
		\nonumber	\mathcal{E}_{\bf d}&:=\sum_{i=1}^n \Big(|M^{\dagger}\frac{\pr M(\Psi_{\mathbb{QS}})}{\pr {\bf d}_i}M^{\dagger}||{\bf d}_i|+|M^{\dagger}{M^{\dagger}}^{\top} \Big(\frac{\pr M(\Psi_{\mathbb{QS}})}{\pr {\bf d}_i}\Big)^{\top} \E_M||{\bf d}_i| +|\F_M\Big(\frac{\pr M(\Psi_{\mathbb{QS}})}{\pr {\bf d}_i}\Big)^{\top} {M^{\dagger}}^{\top} M^{\dagger} ||{\bf d}_i|\Big)\\ 
		& =\sum_{i=1}^n \Big(|M^{\dagger} e_i^{n}(e_i^n)^{\top} M^{\dagger}||{\bf d}_i|+|M^{\dagger}{M^{\dagger}}^{\top} e_i^{n}(e_i^{n})^{\top} \E_M||{\bf d}_i|+|\F_M e_i^{n}(e_i^{n})^{\top} {M^{\dagger}}^{\top} M^{\dagger} ||{\bf d}_i|\Big).
	\end{align}
	Using $(\ref{Eq35})$ in $(\ref{eq53}),$ we deduce
	\begin{align*}
		\mathcal{E}_{\bf d}&=\sum_{i=1}^n \Big(|M^{\dagger} (:,i)||{\bf d}_i|| M^{\dagger}(i,:)|+|M^{\dagger}{M^{\dagger}}^{\top}(:,i)| |{\bf d}_i| | \E_M(i,:)|+|\F_M(:,i) | |{\bf d}_i|| {M^{\dagger}}^{\top} M^{\dagger}(i,:) |\Big)\\
		&=|M^{\dagger}| |\Da_M||M^{\dagger}|+|M^{\dagger}{M^{\dagger}}^{\top}| |\Da_M||\E_M|+|\F_M| |\Da_M| |{M^{\dagger}}^{\top} M^{\dagger}|.
	\end{align*}
	Similarly, for $\{ \a_i\}_{i=2}^n$ and using \eqref{itm:2} of Lemma \ref{lm51}, we get
	\begin{align*}
		\mathcal{E}_{\a}&:=\sum_{i=2}^n \Big(|M^{\dagger}\frac{\pr M(\Psi_{\mathbb{QS}})}{\pr {\a}_i}M^{\dagger}||{\a}_i|+|M^{\dagger}{M^{\dagger}}^{\top} \Big(\frac{\pr M(\Psi_{\mathbb{QS}})}{\pr {\a}_i}\Big)^{\top} \E_M||{\a}_i|\\
		& \quad \quad \quad \quad+|\F_M\Big(\frac{\pr M(\Psi_{\mathbb{QS}})}{\pr {\a}_i}\Big)^{\top} {M^{\dagger}}^{\top} M^{\dagger} ||{\a}_i|\Big)\\
		&=\sum_{i=2}^n \Big( |M^{\dagger} e_i^{n}\L_M(i,:)M^{\dagger}|+|M^{\dagger}{M^{\dagger}}^{\top}\L_M^{\top}(:,i)(e_i^{n})^{\top}\E_M|+|\F_M \L_M^{\top}(:,i)(e_i^{n})^{\top} {M^{\dagger}}^{\top} M^{\dagger}|\Big)\\
		&=|M^{\dagger}||\L_M M^{\dagger}|+|M^{\dagger}{M^{\dagger}}^{\top} \L_M^{\top}| |\E_M|+|\F_M \L_M^{\top}||{M^{\dagger}}^{\top}M^{\dagger}|.
	\end{align*}
	For the parameters $\{\b_i\}_{i=1}^{n-1}$ and using \eqref{itm:4} of Lemma \ref{lm51}, we deduce
	\begin{align*}
		\mathcal{E}_{\b}:&=\sum_{i=2}^n \Big( |M^{\dagger}\frac{\pr M(\Psi_{\mathbb{QS}})}{\pr {\b}_i}M^{\dagger}||{\b}_i|+|M^{\dagger}{M^{\dagger}}^{\top} \Big(\frac{\pr M(\Psi_{\mathbb{QS}})}{\pr {\b}_i}\Big)^{\top} \E_M||{\b}_i|\\
		&\quad \quad +|\F_M\Big(\frac{\pr M(\Psi_{\mathbb{QS}})}{\pr {\b}_i}\Big)^{\top} {M^{\dagger}}^{\top} M^{\dagger} ||{\b}_i|\Big)\\
		&=|M^{\dagger}\L_M| |M^{\dagger}|+|M^{\dagger}{M^{\dagger}}^{\top}| | \L_M^{\top}\E_M|+|\F_M| | \L_M^{\top}{M^{\dagger}}^{\top}M^{\dagger}|.
	\end{align*}
	For the parameters $\{\e_i\}_{i=2}^{n-1}$, using \eqref{itm:3} of Lemma \ref{lm51}, we have
	\begin{align*}
		\mathcal{E}_{\e}&:=\sum_{i=2}^{n-1} \Big(|M^{\dagger}\frac{\pr M(\Psi_{\mathbb{QS}})}{\pr {\e}_i}M^{\dagger}||{\e}_i|+|M^{\dagger}{M^{\dagger}}^{\top} \Big(\frac{\pr M(\Psi_{\mathbb{QS}})}{\pr {\e}_i}\Big)^{\top} \E_M||{\e}_i|\\
		& \quad \quad+|\F\Big(\frac{\pr M(\Psi_{\mathbb{QS}})}{\pr {\e}_i}\Big)^{\top} {M^{\dagger}}^{\top} M^{\dagger} ||{\e}_i|\Big)\\
		&=\sum_{i=2}^{n-1}\Big(|M^{\dagger}\mathcal{F}_iM^{\dagger}|+| M^{\dagger}{M^{\dagger}}^{\top}\mathcal{F}_i^{\top} \E_M|+|\F_M\mathcal{F}_i^{\top}{M^{\dagger}}^{\top}M^{\dagger}|\Big).
	\end{align*}
	In a similar approach, for the parameters $\{\f_i\}_{i=1}^{n-1},\, \{\g_i\}_{i=2}^{n-1}$ and $\{\h_i\}_{i=2}^{n-1},$ we get
	\begin{align*}
		\mathcal{E}_{\f}&:=\sum_{i=1}^{n-1}\Big( |M^{\dagger}\frac{\pr M(\Psi_{\mathbb{QS}})}{\pr {\f}_i}M^{\dagger}||{\f}_i|+|M^{\dagger}{M^{\dagger}}^{\top} \Big(\frac{\pr M(\Psi_{\mathbb{QS}})}{\pr {\f}_i}\Big)^{\top} \E_M||{\f}_i|\\
		& \quad \quad+|\F_M\Big(\frac{\pr M(\Psi_{\mathbb{QS}})}{\pr {\f}_i}\Big)^{\top} {M^{\dagger}}^{\top} M^{\dagger} ||{\f}_i|\Big)\\
		&=|M^{\dagger}||\U_MM^{\dagger}|+|M^{\dagger}{M^{\dagger}}^{\top} \U_M^{\top}| |\E_M|+|\F_M \U_M^{\top}||{M^{\dagger}}^{\top}M^{\dagger}|.
	\end{align*}
	\begin{align*}
		\mathcal{E}_{\h}&:=\sum_{i=2}^{n} \Big(|M^{\dagger}\frac{\pr M(\Psi_{\mathbb{QS}})}{\pr {\h}_i}M^{\dagger}||{\h}_i|+|M^{\dagger}{M^{\dagger}}^{\top} \Big(\frac{\pr M(\Psi_{\mathbb{QS}})}{\pr {\h}_i}\Big)^{\top} \E_M||{\h}_i|\\
		&\quad \quad+|\F_M\Big(\frac{\pr M(\Psi_{\mathbb{QS}})}{\pr {\h}_i}\Big)^{\top} {M^{\dagger}}^{\top} M^{\dagger} ||{\h}_i|\Big)\\
		&=|M^{\dagger}\U_M| |M^{\dagger}|+|M^{\dagger}| |{M^{\dagger}}^{\top} \U_M^{\top}\E_M|+|\F_M| | \U_M^{\top}{M^{\dagger}}^{\top}M^{\dagger}|.
	\end{align*}
	\begin{align*}
		\mathcal{E}_{\g}&:=\sum_{i=2}^{n-1} \Big(|M^{\dagger}\frac{\pr M(\Psi_{\mathbb{QS}})}{\pr {\g}_i}M^{\dagger}||{\g}_i|+|M^{\dagger}{M^{\dagger}}^{\top} \Big(\frac{\pr M(\Psi_{\mathbb{QS}})}{\pr {\g}_i}\Big)^{\top} \E_M||{\g}_i|\\
		&\quad \quad+|\F_M\Big(\frac{\pr M(\Psi_{\mathbb{QS}})}{\pr {\g}_i}\Big)^{\top} {M^{\dagger}}^{\top} M^{\dagger} ||{\g}_i|\Big)\\
		&=\sum_{i=2}^{n-1}\Big(|M^{\dagger}\mathcal{G}_iM^{\dagger}|+| M^{\dagger}{M^{\dagger}}^{\top}\mathcal{G}_i^{\top} \E_M|+|\F_M \mathcal{G}_i^{\top}{M^{\dagger}}^{\top}M^{\dagger}|\Big).
	\end{align*}
	Now, it is straightforward from Theorem \ref{Th31} that $$\mathcal{X}^{\dagger}_{\mathbb{QS}}=\mathcal{E}_{\bf d}+\mathcal{E}_{\a}+\mathcal{E}_{\b}+\mathcal{E}_{\e}+\mathcal{E}_{\f}+\mathcal{E}_{\g}+\mathcal{E}_{\h},$$
	and hence, the proof is completed. 
\end{proof}
 \begin{remark}
     When $M(\Psi_{\mathbb{QS}})$ is a $n\times n$ nonsingular $\{1,1\}$-QS matrix, we have $\E_M={\bf 0}.$ Then, using the expressions in Theorem \ref{Th32},  and an analogous way to Theorem \ref{Th51}, exact formulae of the structured CNs can be deduced for the inversion of $M(\Psi_{\mathbb{QS}}).$ 
 \end{remark}
For any $\{1,1\}$-QS matrix $M(\Psi_{\mathbb{QS}})$,  it is worth noting that there exist infinitely many QS representations, as indicated by Dopico and Pom$\Acute{e}$s \cite{martinez2016structured}. The next result demonstrates that $\tilde{\M}^{\dagger}(M(\Psi_{\mathbb{QS}}))$ and $\tilde{\C}^{\dagger}((\Psi_{\mathbb{QS}}))$ are independent of the QS representation used.

\begin{proposition}
	For any two representations  $\Psi_{\mathbb{QS}}$ and  $\Psi'_{\mathbb{QS}}$ of a $\{1,1\}$-QS matrix $M\in \R^{m\times n},$ we get
 \begin{equation*}
     \tilde{\M}^{\dagger}(M(\Psi_{\mathbb{QS}}))=\tilde{\M}^{\dagger}(M(\Psi'_{\mathbb{QS}}))  \quad \text{and}\quad \tilde{\C}^{\dagger}(M(\Psi_{\mathbb{QS}}))=\tilde{\C}^{\dagger}(M(\Psi'_{\mathbb{QS}})).
 \end{equation*}
\end{proposition}
\begin{proof}
	Observe that the formulae given by the Theorem \ref{Th51} only depend on the entries of $M,$ $M^{\dagger},$ $\E_M$ and $\F_M,$ but not on the specific selection of the parameter set; hence, the proof follows.
\end{proof}
We provide upper bounds on the structured CNs of the MNLS solution of the LS problem $(\ref{eqn312})$ corresponding to $\{1,1\}$-QS matrices in the next theorem.
\begin{theorem}\label{Th52}
	Let $M(\Psi_{\mathbb{QS}})\in \mathbb{R}^{n\times n}$ be such that $\rank(M(\Psi_{\mathbb{QS}}))=r$ and $b\in \mathbb{R}^n.$ 
	Set $\r:=b-M(\Psi_{\mathbb{QS}})\x.$ Then
	\begin{equation*}
		\tilde{\M}^{\dagger}(M(\Psi_{\mathbb{QS}}),b)= \frac{\|\mathcal{X}^{ls}_{\mathbb{QS}}\|_{\infty}}{\|\x\|_{\infty}} \quad \mbox{and}\quad\tilde{\C}^{\dagger}(M(\Psi_{\mathbb{QS}}),b)=\left\| \frac{\mathcal{X}^{ls}_{\mathbb{QS}}}{\x}\right\|_{\infty},
	\end{equation*}
	where
 \vspace{-1mm}
	\begin{align*}
		\mathcal{X}^{ls}_{\mathbb{QS}}:&=|M^{\dagger}||b|+|M^{\dagger}||\Da_M| |\x|+|M^{\dagger}{M^{\dagger}}^{\top}||\Da_M||\r|+|\F_M||\Da_M||{M^{\dagger}}^{\top}\x|+|M^{\dagger}||\L_M \x|\\
		&+|M^{\dagger}{M^{\dagger}}^{\top}\L_M^{T}|| \r|+|\F_M \L_M^{\top}||{M^{\dagger}}^{\top}\x|+|M^{\dagger}\L_M||\x|+|M^{\dagger}{M^{\dagger}}^{\top}||\L_M^{\top} \r|\\
		&+|\F_M||\L_M^{\top}|{M^{\dagger}}^{\top}\x|+|M^{\dagger}||\U_M \x |+|M^{\dagger}{M^{\dagger}}^{\top}\U_M^{\top}||\r|+|\F_M \U_M^{\top}||{M^{\dagger}}^{\top}\x|\\
		&+|M^{\dagger}\U_M||\x|+|M^{\dagger}{M^{\dagger}}^{\top}||\U_M^{\top}\r|+|\F_M|| \U_M^{\top}{M^{\dagger}}^{\top} \x|\\
		&+\sum_{i=2}^{n-1}\Big(|M^{\dagger}\mathcal{F}_i \x|+|M^{\dagger}{M^{\dagger}}^{\top}\mathcal{F}_i^{\top}\r|+|\F_M\mathcal{F}_i^{\top}{M^{\dagger}}^{\top}\x|\Big)\\
		&+\sum_{j=2}^{n-1}\Big(|M^{\dagger}\mathcal{G}_j \x|+|M^{\dagger}{M^{\dagger}}^{\top}\mathcal{G}_j                                                        ^{\top}\r|+|\F_M \mathcal{G}_j^{\top}{M^{\dagger}}^{\top}\x|\Big).
	\end{align*}
\end{theorem}
\begin{proof}
	The statement can be easily verified using a similar justification to that of the proof of the Theorem \ref{Th51} and using the Theorem \ref{Th33}. Hence, we omit the proof.
\end{proof}
\begin{remark}
	By considering $M(\Psi_{\mathbb{QS}})$ nonsingular, for $A(\Psi_{\mathbb{QS}})\x=b,$ using the expression in Theorem \ref{Th34} and $\r=0,$ we can obtain following expression for the structured MCN of $\x:$ 
	\begin{align*}
 \M^{\dagger}\big(M(\Psi_{\mathbb{QS}}),b\big)=\frac{1}{\|\x\|_{\infty}}&\Vr	|M^{-1}||b|+|M^{-1}||\Da_M||\x|+|M^{-1}||\L_M\x|+|M^{-1}\L_M||\x|+|M^{-1}||\U_M\x|\\
		&+|M^{-1}\U_M||\x|+\sum_{i=2}^{n-1}|M^{-1}\mathcal{F}_i\x|+\sum_{j=2}^{n-1}|M^{-1}\mathcal{G}_j\x|\Vr_{\max}.
	\end{align*}
	This result is the same as obtained by Dopico and Pom$\acute{e}$s \cite{martinez2016structured}.
\end{remark}
\subsection{CNs corresponding to GV representation}\label{ss52}
The	GV representation, proposed initially in \cite{vandebril2005note}, is another essential representation for $\{1,1\}$-QS matrices. This representation is used to enhance the stability of fast algorithms. In this subsection, we first review the GV representation together with its minor variant called GV representation through tangent. Then, discuss the results for structured CNs corresponding to this representation.
\begin{definition}\label{def52}\cite{vandebril2005note}
	Any $M\in \mathbb{R}^{n\times n}$ is classified to be a $\{1,1\}$-QS matrix if it can be represented by the parameter set
	\begin{align}\label{eq55}
		\Psi_{\mathbb{QS}}^{\mathcal{GV}}=\Big[\{p_i,q_i\}_{i=2}^{n-1},\{u_i\}_{i=1}^{n-1},\{{\bf d}_i\}_{i=1}^n,\{v_i\}_{i=1}^{n-1},\{r_i,s_i\}_{i=2}^{n-1}\Big]^{\top}\in \R^{7n-10}, 
	\end{align}
	satisfying the following properties,
	\begin{enumerate}
		\item $\{p_i,q_i\}$ is a cosine-sine pair with $p_i^2+q_i^2=1,$ for every $i\in \{2:n-1\},$
		\item $\{u_i\}_{i=1}^{n-1}, \, \{{\bf d}_i\}_{i=1}^n, \, \text{and} \, \, \{v_i\}_{i=1}^{n-1}$ are independent parameters,
		\item $\{r_i,s_i\}$ is a cosine-sine pair with $r_i^2+s_i^2=1,$ for every $i\in \{2:n-1\},$
	\end{enumerate}
	as follows:
	{\footnotesize \begin{align*}
			M=\bmatrix{
				{\bf d}_1 &  v_1r_2 & v_1s_2r_3 & \ldots & v_1s_2\ldots s_{n-2}r_{n-1}& v_1s_2\ldots s_{n-1}\\
				p_2u_1 & {\bf d}_2 & v_2r_3 &\ldots & v_2s_3\ldots s_{n-2}r_{n-1}& v_2s_3\ldots s_{n-1}\\
				p_3q_2u_1 & p_3u_2 &{\bf d}_3 & \ldots & v_3s_4\ldots s_{n-2}r_{n-1}& v_3s_4\ldots s_{n-1}\\
				\vdots & \vdots &\vdots & \ddots & \vdots & \vdots\\
				p_{n-1}q_{n-2}\ldots q_2u_1 & p_{n-1}q_{n-2}\ldots q_{3}u_2& p_{n-1}q_{n-2}\ldots q_{4}u_3& \cdots & {\bf d}_{n-1} & v_{n-1}\\
				q_{n-1}\ldots q_2u_1 & q_{n-1}\ldots q_{3}u_2& q_{n-1}\ldots q_{4}u_3& \cdots  & u_{n-1}& {\bf d}_{n} }.
	\end{align*}}
\end{definition}

Note that the $\Psi_{\mathbb{QS}}^{\mathcal{GV}}$ is a special case of $\Psi_{\mathbb{QS}}$ by considering $ \{\a_i,\e_i\}_{i=2}^{n-1}=\{p_i,q_i\}_{i=2}^{n-1},$ $\{\b_i\}_{i=1}^{n-1}=\{u_i\}_{i=1}^{n-1}$, $\{{\bf \d}_i\}_{i=1}^{n}=\{{\bf d}_i\}_{i=1}^{n},$ $\{\f_i\}_{i=1}^{n-1}=\{v_i\}_{i=1}^{n-1},$ $\{\g_i,\h_i\}_{i=2}^{n-1}=\{s_i,r_i\}_{i=2}^{n-1},$ and $\a_n=\h_n=1$ with additional conditions on the parameters.
Since the parameters $p_i$ and $q_i$ are dependent, arbitrary perturbation to $\Psi_{\mathbb{QS}}^{\mathcal{GV}}$ will destroy the cosine-sine pairs and the same is true for $r_i$ and $s_i.$ Thus, it will be more sensible to restrict the perturbation that preserves the cosine-sine pair. Consequently, Dopico and Pom$\acute{e}$s in \cite{dopico2016structured} introduced a new representation called GV representation through tangent using their tangents.
\begin{definition}
	For the GV representation $\Psi_{\mathbb{QS}}^{\mathcal{GV}}$ as in $(\ref{eq55})$, the GV representation through  tangent is defined as 
	\begin{align}\label{eq56}
		\Psi_{\mathcal{GV}}=\big[\{t_i\}_{i=2}^{n-1}, \{u_i\}_{i=1}^{n-1},\{{\bf d_i}\}_{i=1}^n, \{v_i\}_{i=1}^{n-1},\{w_i\}_{i=2}^{n-1}\big]^{\top}\in \R^{5n-6},
	\end{align}
	where $p_i=\frac{1}{\sqrt{1+t_i^2}},\, q_i=\frac{t_i}{\sqrt{1+t_i^2}}\,\, \text{and} \, \, r_i=\frac{1}{\sqrt{1+w_i^2}},\,  s_i=\frac{w_i}{\sqrt{1+w_i^2}} ,$ for \, $i=2: n-1.$
\end{definition}
We employ the notation $M(\Psi_{\mathcal{GV}})$ to refer a $\{1,1\}$-QS matrix parameterized by the set $\Psi_{\mathcal{GV}}.$ 
The derivative expressions corresponding to the parameters in the representation $\Psi_ {\mathcal{GV}}$ are revisited in the  next lemma:
\begin{lemma}\label{lm52}\cite{dopico2016structured}
	Let $M(\Psi_{\mathcal{GV}})\in \R^{n\times n}$ having $\rank(M(\Psi_{\mathcal{GV}}))=r.$ Then each entry of $M(\Psi_{\mathcal{GV}})$ is differentiable functions of the elements in $\Psi_{\mathcal{GV}},$ and 
 \begin{enumerate}
     \item  $t_i\frac{\pr M(\Psi_{\mathcal{GV}})}{\pr t_i}=\bmatrix{
			{\bf 0} & {\bf 0}\\
			-q_i^2M(\Psi_{\mathcal{GV}})(i,1:i-1)& {\bf 0} \\
			p_i^2 M(\Psi_{\mathcal{GV}})(i+1:n,1:i-1) & {\bf 0}
		}:=\mathcal{K}_i,$ for $ i=2:n-1. $\label{itm1a}
  \item $w_i \frac{\pr M(\Psi_{\mathcal{GV}})}{\pr w_i}=\bmatrix{
			{\bf 0} & -s_i^2M(\Psi_{\mathcal{GV}})(1:i-1,i) & r_i^2 M(\Psi_{\mathcal{GV}})(1:i-1,i+1:n)\\
			{\bf 0} & {\bf 0} & {\bf 0}
		}:=\mathcal{L}_i,$ for $ i=2:n-1. $ \label{itm2a}
 \end{enumerate}
\end{lemma}
{\bf Note: } Partial derivative expressions corresponding to the parameters $\{u_i\}_{i=1}^{n-1}, \{{\bf d}_i\}_{i=1}^n$ and $ \{v_i\}_{i=1}^{n-1}$ are same as the expression for the parameters $\{\b_i\}_{i=1}^{n-1}, \{{\bf d}_i\}_{i=1}^n $ and $\{\f_i\}_{i=1}^{n-1},$ respectively, given in the Lemma \ref{lm51}.

In Theorem \ref{Th53}, we discuss computationally feasible upper bounds for the structured CNs introduced in Definition $\ref{def31}$ corresponding to the representation $\Psi_{\mathcal{GV}}$.
\begin{theorem}\label{Th53}
	Let $M(\Psi_{\mathcal{GV}})\in \mathbb{R}^{n\times n}$ with $\rank(M(\Psi_{\mathcal{GV}}))=r,$ then
	\begin{eqnarray*}
\tilde{\M}^{\dagger}\big(M(\Psi_{\mathcal{GV}})\big)=\frac{\|\mathcal{X}^{\dagger}_{\mathcal{GV}}\|_{\max}}{\|M^{\dagger}\|_{\max}}\quad  \mbox{and}\quad\tilde{\C}^{\dagger}\big(M(\Psi_{\mathcal{GV}})\big)=\Vr \frac{\mathcal{X}^{\dagger}_{\mathcal{GV}}}{M^{\dagger}}\Vr_{\max},
	\end{eqnarray*}
	where
	\begin{align*}
\mathcal{X}^{\dagger}_{\mathcal{GV}}:&=|M^{\dagger}||\Da_M||M^{\dagger}|+|M^{\dagger}{M^{\dagger}}^{\top}||\Da_M||\E_M|+|\F_M||\Da_M||{M^{\dagger}}^{\top}M^{\dagger}|+|M^{\dagger}\L_M||M^{\dagger}|\\
		&+|M^{\dagger}{M^{\dagger}}^{\top}||\L_M^{\top}\E_M|+|\F_M||\L_M^{\top}{M^{\dagger}}^{\top}M^{\dagger}|+|M^{\dagger}||\U_MM^{\dagger}|+|M^{\dagger}{M^{\dagger}}^{\top}\U_M^{\top}||\E_M|\\
		&+|\F_M \U_M^{\top}||{M^{\dagger}}^{\top}M^{\dagger}|+\sum_{i=2}^{n-1}\Big(|M^{\dagger}\mathcal{K}_iM^{\dagger}|+|M^{\dagger}{M^{\dagger}}^{\top}\mathcal{K}_i^{\top} \E_M|+|\F_M \mathcal{K}_i^{\top}{M^{\dagger}}^{\top}M^{\dagger}|\Big)\\
		&+\sum_{j                    =2}^{n-1}\Big(|M^{\dagger}\mathcal{L}_jM^{\dagger}|+|M^{\dagger}{M^{\dagger}}^{\top}\mathcal{L}_j^{\top} \E_M|+|\F_M \mathcal{L}_j^{\top}{M^{\dagger}}^{\top}M^{\dagger}|\Big).
	\end{align*}
	Here, $\mathcal{K}_i$ and $\mathcal{L}_i$ are defined as in Lemma \ref{lm52}.
\end{theorem}
\begin{proof}
	To prove the claim, we need to find the contribution of each subset of parameters in $\Psi_{\mathcal{GV}}$ to the expressions given in Theorem \ref{Th31}. Since the derivative expressions for the parameters $\{u_i\}_{i=1}^{n-1}, \{{\bf d}_i\}_{i=1}^n$ and $ \{v_i\}_{i=1}^{n-1}$ are same as the expressions for the parameters $\{\b_i\}_{i=1}^{n-1}, \{{\bf d}_i\}_{i=1}^n$ and $\{\f_i\}_{i=1}^{n-1},$ respectively, given in the Lemma \ref{lm51}, contribution for these parameters in the expressions  given in Theorem \ref{Th31} are same as $ \mathcal{E}_{\b},$ $ \mathcal{E}_{\bf d}$ and $ \mathcal{E}_{\f},$ respectively. Therefore, we set  $\mathcal{E}_u:=\mathcal{E}_{\b}$ and $\mathcal{E}_v:=\mathcal{E}_{\f}.$
	
	For the parameters $\{t_i\}_{i=2}^{n-1},$ we have
	\begin{align*}
		\mathcal{E}_{t}&:=\sum_{i=2}^{n-1} \Big(|M^{\dagger}\frac{\pr M(\Psi_{\mathcal{GV}})}{\pr {t}_i}M^{\dagger}||{t}_i|+|M^{\dagger}{M^{\dagger}}^{\top} \Big(\frac{\pr M(\Psi_{\mathcal{GV}})}{\pr {t}_i}\Big)^{\top} \E_M||{t}_i| +|\F_{M}\Big(\frac{\pr M(\Psi_{\mathcal{GV}})}{\pr {t}_i}\Big)^{\top} {M^{\dagger}}^{\top} M^{\dagger} ||{t}_i|\Big)\\
		&=\sum_{i=2}^{n-1}\Big(|M^{\dagger}\mathcal{K}_iM^{\dagger}|+| M^{\dagger}{M^{\dagger}}^{\top}\mathcal{K}_i^{\top} \E_M|+|\F_M \mathcal{K}_i^{\top}{M^{\dagger}}^{\top}M^{\dagger}|\Big).
	\end{align*}
	For the parameters $\{{w}_i\}_{i=2}^{n-1},$ we have
	\begin{align*}
		\mathcal{E}_{w}&:=\sum_{i=2}^{n-1} \Big(|M^{\dagger}\frac{\pr M(\Psi_{\mathcal{GV}})}{\pr {w}_i}M^{\dagger}||{w}_i|+|M^{\dagger}{M^{\dagger}}^{\top} \Big(\frac{\pr M(\Psi_{\mathcal{GV}})}{\pr {w}_i}\Big)^{\top} \E_M||{w}_i|+|\F\Big(\frac{\pr M(\Psi_{\mathcal{GV}})}{\pr {w}_i}\Big)^{\top} {M^{\dagger}}^{\top} M^{\dagger} ||{w}_i|\Big)\\
		&=\sum_{i=2}^{n-1}\Big(|M^{\dagger}\mathcal{L}_iM^{\dagger}|+| M^{\dagger}{M^{\dagger}}^{\top}\mathcal{L}_i^{\top} \E_M|+|\F_M \mathcal{L}_i^{\top}{M^{\dagger}}^{\top}M^{\dagger}|\Big).
	\end{align*}
	Now, from Theorem \ref{Th31}, we have $\mathcal{X}^{\dagger}_{\mathcal{GV}}=\mathcal{E}_{\bf d}+\mathcal{E}_{u}+\mathcal{E}_{v}+\mathcal{E}_{t}+\mathcal{E}_{w},$
	and hence, the proof follows. 
\end{proof}
The bounds of the structured CNs for $\x$ with respect to the representation $\Psi_{\mathcal{GV}}$  are given in the next theorem. 

\begin{theorem}\label{Th54}
	Let $M(\Psi_{\mathcal{GV}})\in \mathbb{R}^{n\times n}$ with $\rank(M(\Psi_{\mathcal{GV}}))=r.$
	Set $\r:=b-M(\Psi_{\mathcal{GV}})\x,$ then
	\begin{eqnarray*}
\tilde{\M}^{\dagger}\big(M(\Psi_{\mathcal{GV}}),b\big)=\frac{\|\mathcal{X}^{ls}_{\mathcal{GV}}\|_{\infty}}{\|\x\|_{\infty}} \quad  \mbox{and}\quad\tilde{\C}\big(M(\Psi_{\mathcal{GV}}),b\big)=\Bigg\Vert \frac{\mathcal{X}^{ls}_{\mathcal{GV}}}{\x}\Bigg\Vert_{\infty},
	\end{eqnarray*}
	where
	\begin{align*}
\mathcal{X}^{ls}_{\mathcal{GV}}:&=|M^{\dagger}||\Da_M||\x|+|M^{\dagger}{M^{\dagger}}^{\top}||\Da_M||\r|+|\F_M||\Da_M||{M^{\dagger}}^{\top}\x|+|M^{\dagger}\L_M||\x|\\
		&+|M^{\dagger}{M^{\dagger}}^{\top}||\L_M^{\top} \r|+|\F_M||\L_M^{\top}{M^{\dagger}}^{\top} \x|+|M^{\dagger}||\U_M\x|+|M^{\dagger}{M^{\dagger}}^{\top}\U_M^{\top}||\r|\\
		&+|\F_M \U_M^{\top}||{M^{\dagger}}^{\top}\x|+\sum_{i=2}^{n}\Big(|M^{\dagger}\mathcal{K}_i\x|+|M^{\dagger}{M^{\dagger}}^{\top}\mathcal{K}_i^{\top} \r|+|\F_M \mathcal{K}_i^{\top}{M^{\dagger}}^{\top}\x|\Big)\\
&+\sum_{i=2}^{n}\Big(|M^{\dagger}\mathcal{L}_i\x|+|M^{\dagger}{M^{\dagger}}^{\top}\mathcal{L}_i^{\top} \r|+|\F_M \mathcal{L}_i^{\top}{M^{\dagger}}^{\top}\x|\Big)+|M^{\dagger}||b|.
	\end{align*}
	Here, $\mathcal{K}_i$ and $\mathcal{L}_i$ are defined as in Lemma \ref{lm52}.
\end{theorem}
\begin{proof}
	The proof follows 
 by using the similar proof technique of Theorem \ref{Th53}.   Hence, we omit it.
\end{proof}
\subsection{Comparisons between different CNs for $\{1,1\}$-QS matrices}\label{ss53}
We compare structured and unstructured CNs for the M-P inverse in Proposition \ref{pr53} and the MNLS solution in Proposition \ref{pr54} for $\{1,1\}$-QS matrix. For unstructured CNs, we use the expressions given in Corollary \ref{coro31} and Corollary \ref{coro34}. Relationships between the CNs for representations $\Psi_{\mathbb{QS}}$ and $\Psi_{\mathcal{GV}}$ are also investigated in Propositions $\ref{pr55}$ and $\ref{pr56}.$

The next result describes that structured CNs for the parameter set $\Psi_{\mathbb{QS}}$ are smaller than unstructured ones for the M-P inverse up to an order of $n$.
\begin{proposition}\label{pr53}
	Let $M(\Psi_{\mathbb{QS}})\in \mathbb{R}^{n\times n}$ be such that $\rank(M(\Psi_{\mathbb{QS}}))=r,$ then we get the following relations
	\begin{equation*}
		\tilde{\M}^{\dagger}\left(M(\Psi_{\mathbb{QS}})\right)\leq n  \,\,  \tilde{\M}^{\dagger}(M)  \, \, \mbox{and}\, \, \, \tilde{\C}^{\dagger}\left(M(\Psi_{\mathbb{QS}})\right)\leq  n \,  \,\tilde{\C}^{\dagger}(M).
	\end{equation*}
\end{proposition}
\begin{proof}
	Using the properties of absolute values and Theorem \ref{Th51}, we have
	\begin{align*}
\mathcal{X}^{\dagger}_{\mathbb{QS}}&\leq|M^{\dagger}||\Da_M||M^{\dagger}|+|M^{\dagger}{M^{\dagger}}^{\top}||\Da_M||\E_M|+|\F_M||\Da_M||{M^{\dagger}}^{\top}M^{\dagger}|\\
		&+2|M^{\dagger}||\L_M||M^{\dagger}|+2|M^{\dagger}{M^{\dagger}}^{\top}| |\L_M^{T}|| \E_M| +2|\F_M|| \L_M^{\top}||{M^{\dagger}}^{\top}M^{\dagger}|\\
		&+2|M^{\dagger}||\U_M ||M^{\dagger} |
		+2|M^{\dagger}{M^{\dagger}}^{\top}||\U_M^{\top}||\E_M|+2|\F_M|| \U_M^{\top}||{M^{\dagger}}^{\top}M^{\dagger}|\\
		&+\sum_{i=2}^{n-1}\Big(|M^{\dagger}||\mathcal{F}_i||M^{\dagger}|+|M^{\dagger}{M^{\dagger}}^{\top}||\mathcal{F}_i^{\top}||\E_M|
		+|\F_M||\mathcal{F}_i^{\top}||{M^{\dagger}}^{\top}M^{\dagger}|\Big)\\
		&+\sum_{i=2}^{n-1}\Big(|M^{\dagger}| |\mathcal{G}_i||M^{\dagger}|+|M^{\dagger}{M^{\dagger}}^{\top}||\mathcal{G}_i^{\top}||\E_M|+|\F_M|| \mathcal{G}_i^{\top}||{M^{\dagger}}^{\top}M^{\dagger}|\Big).
	\end{align*}
	Using $|\mathcal{F}_i|\leq |\L_M|$ and $|\mathcal{G}_i|\leq |\U_M|$, we get
	\begin{equation*}
		\mathcal{X}^{\dagger}_{\mathbb{QS}} \leq n \,\Big(|M^{\dagger}||M||M^{\dagger}|+|M^{\dagger}{M^{\dagger}}^{\top}||M^{\top}||\E_M|+|\F_M||M^{\top}||{M^{\dagger}}^{\top}M^{\dagger}|\Big).
	\end{equation*}
	Therefore, the desired relations can be obtained from Theorem \ref{Th51} and Corollary \ref{coro31}.
\end{proof}
A similar type of result also holds for the LS problem, which is given next. We remove the proof since it is analogous to Proposition \ref{pr53}.
\begin{proposition}\label{pr54}
	Let $M(\Psi_{\mathbb{QS}})\in \R^{n\times n}$ be as in Proposition \ref{pr53} and $b\in \mathbb{R}^m.$ Then, we get the following relations
	\begin{equation*}
		\tilde{\M}^{\dagger}(M(\Psi_{\mathbb{QS}}),b)\leq n \, \,\,  \tilde{\M}^{\dagger}(M,b) \quad \mbox{and}\quad \tilde{\C}^{\dagger}(M(\Psi_{\mathbb{QS}}),b)\leq \, \, n \, \,\tilde{\C}^{\dagger}(M,b).
	\end{equation*}
\end{proposition}

Next result discuss about the relationship between the $\tilde{\M}^{\dagger}(M(\Psi_{\mathbb{QS}}))$ with $\tilde{\M}^{\dagger}(M(\Psi_{\mathcal{GV}}))$ and $\tilde{\C}^{\dagger}(M(\Psi_{\mathbb{QS}}))$ with $\tilde{\C}^{\dagger}(M(\Psi_{\mathcal{GV}})).$
\begin{proposition}\label{pr55}
	For the	representations $\Psi_{\mathbb{QS}}$ and $\Psi_{\mathcal{GV}}$ of a $\{1,1\}$-QS matrix $M\in \mathbb{R}^{n\times n}$ with $\rank(M)=r,$ following holds:
	\begin{equation*}
		\tilde{\M}^{\dagger}(M(\Psi_{\mathcal{GV}}))\leq \tilde{\M}^{\dagger}(M(\Psi_{\mathbb{QS}})) \quad \mbox{and}\quad	\tilde{\C}^{\dagger}(M(\Psi_{\mathcal{GV}}))\leq \tilde{\C}^{\dagger}(M(\Psi_{\mathbb{QS}})).
	\end{equation*} 
\end{proposition}
\begin{proof}
	The proof will be followed by observing that
	\begin{align*}
		\mathcal{K}_i=\bmatrix{
			{\bf 0} &{\bf 0}\\-q_i^2M(i,1:i-1)  & {\bf 0}\\ p_i^2M(i+1:n,1:i-1) & {\bf 0}
		}&= -q_i^2\bmatrix{
			{\bf 0} &{\bf 0}\\M(i,1:i-1)  & {\bf 0}\\ {\bf 0}& {\bf 0}
		}+p_i^2\bmatrix{
			{\bf 0} &{\bf 0}\\{\bf 0}  & {\bf 0}\\ M(i+1:n,1:i-1) & {\bf 0}
		}\\
		& =-q_i^2\, e_i^m\L_M(i,:)+p_i^2\,\mathcal{F}_i.
	\end{align*}
	Now, using the properties $|p_i|^2\leq 1$ and $|q_i|^2\leq 1,$ and \eqref{Eq35}, we obtain \begin{align*}
		\nonumber \hspace{2mm}\sum_{i=2}^{n-1}\Big(|M^{\dagger}&\mathcal{K}_iM^{\dagger}|+|M^{\dagger}{M^{\dagger}}^{\top}\mathcal{K}_i^{\top}\E_M|+|\F_M\mathcal{K}_i^{\top}{M^{\dagger}}^{\top}M^{\dagger}|\Big)\leq |M^{\dagger}||\L_MM^{\dagger}|+|M^{\dagger}{M^{\dagger}}^{\top}\L_M^{\top}||\E_M|\\
		&+|\F_M \L_M^{\top}||{M^{\dagger}}^{\top}M^{\dagger}|
		+\sum_{i=2}^{n-1}\Big(|M^{\dagger}\mathcal{F}_iM^{\dagger}|+|M^{\dagger}{M^{\dagger}}^{\top}\mathcal{F}_i^{\top}\E_M|+|\F_M\mathcal{F}_i^{\top}{M^{\dagger}}^{\top}M^{\dagger}|\Big).
	\end{align*}
	Similarly, we can write 
	$	\mathcal{L}_i =-s_i^2\, \U_M(:,i)(e_i^n)^{\top}+r_i^2\,\mathcal{G}_i.$
	Therefore, using $|s_i|^2\leq 1$ and $|r_i|^2\leq 1,$ and \eqref{Eq35}, we get
	\begin{align*}
		\nonumber \hspace{2mm}\sum_{i=2}^{n-1}\Big(|M^{\dagger}&\mathcal{L}_iM^{\dagger}|+|M^{\dagger}{M^{\dagger}}^{\top}\mathcal{L}_i^{\top}\E_M|+|\F_M\mathcal{L}_i^{\top}{M^{\dagger}}^{\top}M^{\dagger}|\Big)\leq |M^{\dagger}\U_M||M^{\dagger}|+|M^{\dagger}{M^{\dagger}}^{\top}||\U_M^{\top}\E_M|\\
		&+|\F_M|| \U_M^{\top}{M^{\dagger}}^{\top}M^{\dagger}|+\sum_{i=2}^{n-1}\Big(|M^{\dagger}\mathcal{G}_iM^{\dagger}|+|M^{\dagger}{M^{\dagger}}^{\top}\mathcal{G}_i^{\top}\E_M|+|\F_M\mathcal{G}_i^{\top}{M^{\dagger}}^{\top}M^{\dagger}|\Big).
	\end{align*}
	Hence, we get the desired relations from the above two inequalities, expressions from the Theorems  \ref{Th51} and \ref{Th53}.
\end{proof}
Proposition \ref{pr56} provides the relationship between CNs for LS problem $(\ref{eqn312})$ for any $\{1,1\}$-QS matrix corresponding to the parameter sets $\Psi_{\mathbb{QS}}$ and $\Psi_{\mathcal{GV}}.$
\begin{proposition}\label{pr56}
	For the	representations $\Psi_{\mathbb{QS}}$ and $\Psi_{\mathcal{GV}}$ of a $\{1,1\}$-QS matrix $M\in \mathbb{R}^{n\times n}$ having rank $r$ and $b\in \mathbb{R}^n,$ following holds:
	\begin{eqnarray*}
		\tilde{\M}^{\dagger}(M(\Psi_{\mathcal{GV}}),b)\leq \tilde{\M}^{\dagger}(M(\Psi_{\mathbb{QS}}),b) \quad \text{and}\quad
		\tilde{\C}^{\dagger}(M(\Psi_{\mathcal{GV}}),b)\leq \tilde{\C}^{\dagger}(M(\Psi_{\mathbb{QS}}),b).		
	\end{eqnarray*} 
\end{proposition}
\subsection{The structured effective CNs}
The expressions in Theorems \ref{Th51} and \ref{Th52} can be computationally very expensive for large matrices due to the involvement of two sums. The effective CN for $\{1,1\}$-QS matrices was initially considered by Dopico and Pom$\acute{e}$s in \cite{dopico2016structured,martinez2016structured} for eigenvalue problem and linear system to reduce the computation complexity. In a similar fashion to avoid such problems, we propose in Definition \ref{def54}, structured effective CNs $\tilde{\M}^{\dagger}_{ f}(M(\Psi_{\mathbb{QS}}))$  and $\tilde{\C}^{\dagger}_{f}(M(\Psi_{\mathbb{QS}})),$ which have similar contribution as  $\tilde{\M}^{\dagger}(M(\Psi_{\mathbb{QS}}))$ and $\tilde{\C}^{\dagger}(M(\Psi_{\mathbb{QS}})),$ respectively.
\begin{definition}\label{def54}
	Let $M(\Psi_{\mathbb{QS}})\in \mathbb{R}^{n\times n}$ having $\rank(M(\Psi_{\mathbb{QS}}))=r.$ Then for $M^{\dagger}(\Psi_{\mathbb{QS}}),$ we define structured  effective MCN and CCN as
	\begin{eqnarray*}
		\tilde{\M}^{\dagger}_{f}(M(\Psi_{\mathbb{QS}}))&:=\frac{\|\mathcal{X}^{\dagger}_{f,\mathbb{QS}}\|_{\max}}{\|M^{\dagger}\|_{\max}}\quad \mbox{and}\quad  \tilde{\C}^{\dagger}_{f}(M(\Psi_{\mathbb{QS}})):=\left\|\frac{\mathcal{X}^{\dagger}_{f,\mathbb{QS}}}{M^{\dagger}}\right\|_{\max},
	\end{eqnarray*}
	where   \begin{align*}
		\mathcal{X}^{\dagger}_{f,\mathbb{QS}}:&=	|M^{\dagger}||\Da_M||M^{\dagger}|+|M^{\dagger}{M^{\dagger}}^{\top}||\Da_M||\E_M|+|\F_M||\Da_M||{M^{\dagger}}^{\top}M^{\dagger}|+|M^{\dagger}||\L_MM^{\dagger}|\\
		&+|M^{\dagger}{M^{\dagger}}^{\top}\L_M^{T}|| \E_M| +|\F_M\L_M^{\top}||{M^{\dagger}}^{\top}M^{\dagger}|+|M^{\dagger}\L_M||M^{\dagger}|+|M^{\dagger}{M^{\dagger}}^{\top}||\L_{M}^{\top}\E_M|\\
		&+|\F_M||\L_M^{\top}{M^{\dagger}}^{\top}M^{\dagger}|+|M^{\dagger}||\U_M M^{\dagger} |+|M^{\dagger}{M^{\dagger}}^{\top}\U_M^{\top}||\E_M|+|\F_M \U_M^{\top}||{M^{\dagger}}^{\top}M^{\dagger}|\\
		&+|M^{\dagger}\U_M||M^{\dagger}|+|M^{\dagger}{M^{\dagger}}^{\top}||\U_M^{\top}\E_M|+|\F_M|| \U_M^{\top}{M^{\dagger}}^{\top}M^{\dagger}|.
	\end{align*}
\end{definition}
The following theorem demonstrates that the contribution of the sum terms in the expression of  $\tilde{\M}^{\dagger}(M(\Psi_{\mathbb{QS}}))$  and $\tilde{\C}^{\dagger}(M(\Psi_{\mathbb{QS}}))$  are negligible and reliably estimated up to a multiple of $n.$
\begin{theorem}\label{Th9}
	Under the same hypothesis as in Definition \ref{def54}, following relations holds 
	\begin{align*}
		&	\tilde{\M}^{\dagger}_{f}(M(\Psi_{\mathbb{QS}}))\leq  \tilde{\M}^{\dagger}(M(\Psi_{\mathbb{QS}}))\leq (n-1)\, \tilde{\M}^{\dagger}_f(M(\Psi_{\mathbb{QS}})),\\
		&\tilde{\C}^{\dagger}_{f}(M(\Psi_{\mathbb{QS}}))\leq  \tilde{\C}^{\dagger}(M(\Psi_{\mathbb{QS}}))\leq (n-1)\, \tilde{\C}^{\dagger}_f(M(\Psi_{\mathbb{QS}})).
	\end{align*}
\end{theorem}
\begin{proof}
	The inequalities on the left side are trivial. For inequalities on the right side, we set 
	\begin{align*}
		\bmatrix{
			{\bf 0} & {\bf 0}\\
			 M(\Psi_{\mathbb{QS}})(i+1:n,1:i-1) & {\bf 0}}
		&=\bmatrix{
			{\bf 0} \\	\L_M(i+1:n,:)}+ \bmatrix{
			{\bf 0} & {\bf 0}\\
			{\bf 0} &	-\L_M(i+1:n,i:n)}.
	\end{align*}
	By using the above, we get
	\begin{align*}
		\Big(|M^{\dagger}\mathcal{F}_iM^{\dagger}|&+|M^{\dagger}{M^{\dagger}}^{\top}\mathcal{F}_i^{\top}\E_M|+|\F_M\mathcal{F}_i^{\top}{M^{\dagger}}^{\top}M^{\dagger}|\Big)\leq |M^{\dagger}||\L_MM^{\dagger}|+|M^{\dagger}{M^{\dagger}}^{\top}\L_M^{T}|| \E_M| \\
		&+|\F_M \L_M^{\top}||{M^{\dagger}}^{\top}M^{\dagger}|+|M^{\dagger}\L_M||M^{\dagger}|+|M^{\dagger}{M^{\dagger}}^{\top}||\L_{M}^{\top}\E_M|+|\F_M||\L_M^{\top}|{M^{\dagger}}^{\top}M^{\dagger}|.
	\end{align*} 
	Again, 
	$$\bmatrix{
		{\bf 0} & M(\Psi_{\mathbb{QS}})(1:i-1,i+1:n) \\
		{\bf 0} & {\bf 0}
	}=\bmatrix{
		{\bf 0} & \U_M(:,i+1:n)
	}+\bmatrix{
		{\bf 0} &  {\bf 0}\\
		{\bf 0} & -\U_M(i:n,i+1:n)
	}.$$
	By using the above, we get
	\begin{align*}
		\Big(|M^{\dagger}\mathcal{G}_iM^{\dagger}|&+|M^{\dagger}{M^{\dagger}}^{\top}\mathcal{G}_i^{\top}\E_M|+|\F_M\mathcal{G}_i^{\top}{M^{\dagger}}^{\top}M^{\dagger}|\Big)\leq |M^{\dagger}||\U_M M^{\dagger} |+|M^{\dagger}{M^{\dagger}}^{\top}\U_M^{\top}||\E_M|\\
		&+|\F_M \U_M^{\top}||{M^{\dagger}}^{\top}M^{\dagger}|+|M^{\dagger}\U_M||M^{\dagger}|+|M^{\dagger}{M^{\dagger}}^{\top}||\U_M^{\top}\E_M|+|\F_M|| \U_M^{\top}{M^{\dagger}}^{\top}M^{\dagger}|.
	\end{align*}
	Hence, the desired result is straightforward from Definition \ref{def54}.
\end{proof}
Next, we compare the relation among effective CNs to the unstructured CNs given in Corollary \ref{coro31}. 
\begin{proposition}\label{PR56}
	Under the same hypothesis as in Definition \ref{def54}, following relations holds 
	\begin{align*}
		&	\tilde{\M}^{\dagger}_{f}(M(\Psi_{\mathbb{QS}}))\leq  \, 2 \,\,\tilde{\M}^{\dagger}(M) \quad \mbox{and}\quad
		\tilde{\C}^{\dagger}_{f}(M(\Psi_{\mathbb{QS}}))\leq \, 2 \, \,\tilde{\C}^{\dagger}(M).
	\end{align*}
\end{proposition}
\begin{proof}
	By using properties of absolutes values and expression of $\mathcal{X}^{\dagger}_{f,\mathbb{QS}}$ given the Definition \ref{def54}, we have $$\mathcal{X}^{\dagger}_{f,\mathbb{QS}}\leq \, 2\,  \left(\vert M^{\dagger} \vert \vert M \vert \vert M^{\dagger}\vert + |M^{\dagger}{M^{\dagger}} ^{\top}| |M^{\top}| |\E_M|+|\F_M| |M^{\top}| |{M^{\dagger}} ^{\top}M^{\dagger}|\right).$$
	Hence, the desired results attained from Definition \ref{def54}.
\end{proof}
Next, similarly to the Definition \ref{def54}, we define structured effective CNs for the MNLS solution.
\begin{definition}\label{def55}
	Let $M(\Psi_{\mathbb{QS}})\in \mathbb{R}^{n\times n}$ having $\rank(M(\Psi_{\mathbb{QS}}))=r$ and $b\in \R^m.$ Then, for the MNLS solution $\x$, we define structured effective MCN and CCN as follows:
	\begin{eqnarray*}
		\tilde{\M}^{\dagger}_{f}(M(\Psi_{\mathbb{QS}}),b)&:=\frac{\|\mathcal{X}^{ls}_{f,\mathbb{QS}}\|_{\infty}}{\|\x\|_{\infty}} \quad \text{and} \quad  \tilde{\C}^{\dagger}(M(\Psi_{\mathbb{QS}}),b):=\left\|\frac{\mathcal{X}^{ls}_{f,\mathbb{QS}}}{\x}\right\|_{\infty},
	\end{eqnarray*}
	where  \begin{eqnarray*}
		\mathcal{X}^{ls}_{f,\mathbb{QS}}:&=	|M^{\dagger}||\Da_M||\x|+|M^{\dagger}{M^{\dagger}}^{\top}||\Da_M||\r|+|\F_M||\Da_M||{M^{\dagger}}^{\top}\x|+|M^{\dagger}||\L_M\x|\\
		&+|M^{\dagger}{M^{\dagger}}^{\top}\L_M^{T}|| \r|+|\F_M \L_M^{\top}||{M^{\dagger}}^{\top}\x|+|M^{\dagger}\L_M||\x|+|M^{\dagger}{M^{\dagger}}^{\top}||\L_{M}^{\top}\r|\\
		&+|\F_M||\L_M^{\top}{M^{\dagger}}^{\top}\x|+|M^{\dagger}||\U_M \x |
		+|M^{\dagger}{M^{\dagger}}^{\top}\U_M^{\top}||\r|+|\F_M \U_M^{\top}||{M^{\dagger}}^{\top}\x|\\
		&+|M^{\dagger}\U_M||\x|+|M^{\dagger}{M^{\dagger}}^{\top}||\U_M^{\top}\r|+|\F_M|| \U_M^{\top}{M^{\dagger}}^{\top}\x|+|M^{\dagger}||b|.
	\end{eqnarray*}
\end{definition}
Similar results also hold for the MNLS solution as discussed in Theorem \ref{Th9} and Proposition \ref{PR56}.
\section{Numerical experiments}\label{s6}
We reported a few numerical experiments in this section to illustrate the theoretical findings covered in Sections $\ref{sec4}$ and $\ref{sec5}$, respectively. For all numerical computations, we have used MATLAB R2022b, with unit roundoff errors $10^{-16}.$ Also, for structured and unstructured CNs, we give a comparison between their upper bounds. For unstructured CNs, we use the upper bounds for the M-P inverse and the MNLS solution given in Corollary $\ref{coro31}$ and $\ref{coro34},$ respectively. These examples reveal that structured CNs are significantly less than unstructured ones.
\begin{example}\label{ex1}
	Let $\Psi_{\mathbb{CV}}=[ c=[1,1,\frac{1}{2},\frac{-1}{30},\frac{1}{40}], \, d=[12 , -0.75\times 10^7, 25\times 10^3\}]^{\top}\in \R^8$ be the parameter set of a $5\times 6$ CV matrix $M$. For the M-P inverse, we compute the bounds for the structured MCN and CCN using Theorem \ref{Th41}. 	For the MNLS solution, we generate a random vector $b\in \R^5$ in MATLAB by the command $randn$ and use Theorem \ref{Th43} to compute the structured CNs.  The computed results for the bounds of structured and unstructured CNs are listed in Table \ref{tab1}. We observed that the bounds for the structured CNs are of order one, whereas for the unstructured case, they are significantly large.
	\begin{table}[ht!]
		\begin{center}
		\caption{Comparison between upper bounds of structured and unstructured CNs for $M^{\dagger}(\Psi_{\mathbb{CV}})$ and $M^{\dagger}(\Psi_{\mathbb{CV}},b)$   for Example \ref{ex1}.}%
			 \resizebox{14cm}{!}
   {
				\begin{tabular}{@{}cccccccc@{}}
					\toprule
					$\tilde{\M}^{\dagger}(M)$ & $\tilde{\M}^{\dagger}(M(\Psi_{\mathbb{CV}}))$&  $\tilde{\C}^{\dagger}(M)$ & $\tilde{\C}^{\dagger}(M(\Psi_{\mathbb{CV}}))$&$\tilde{\M}^{\dagger}(M,b)$                         & $\tilde{\M}^{\dagger}(M(\Psi_{\mathbb{CV}}),b)$ & $\tilde{\C}^{\dagger}(M,b)$ & $\tilde{\C}^{\dagger}(M(\Psi_{\mathbb{CV}}),b)$ \\
					\midrule
					$1.9309e+04$ &  $8.4149$ & $2.6103e+06$ & $63.7873$ &$3.9137e+04$ & $12.3655 $                                      & $2.1903e+06$        & $12.3655$   \\
					\toprule
			\end{tabular}}
   \label{tab1}
		\end{center}
	\end{table}
 \end{example}
 \begin{example}\label{Ex1}\cite{yang2021accurate}
      \begin{table}[h!]
		\begin{center}
		\caption{Comparison between upper bounds of structured and unstructured CNs for $M^{\dagger}(\Psi_{\mathbb{CV}})$ and $M^{\dagger}(\Psi_{\mathbb{CV}},b)$   for Example \ref{Ex1}.}%
			\resizebox{15cm}{!}{
				\begin{tabular}{@{}cccccccc@{}}
					\toprule
			$\tilde{\M}^{\dagger}(M)$ & $\tilde{\M}^{\dagger}(M(\Psi_{\mathbb{CV}}))$&  $\tilde{\C}^{\dagger}(M)$ & $\tilde{\C}^{\dagger}(M(\Psi_{\mathbb{CV}}))$&$\tilde{\M}^{\dagger}(M,b)$                         & $\tilde{\M}^{\dagger}(M(\Psi_{\mathbb{CV}}),b)$ & $\tilde{\C}^{\dagger}(M,b)$ & $\tilde{\C}^{\dagger}(M(\Psi_{\mathbb{CV}}),b)$ \\
					\midrule
					$1.1568e+05$&	$8.5419e+01$ &	$5.3826e+07$&	$3.3542e+04$&	$1.8008e+05	$& $8.9389e+01$&	$3.6180e+05$ &	$1.7959e+02$  \\
					\toprule
			\end{tabular}}
   \label{Tab1}
		\end{center}
  
	\end{table}
 Let $M\in \R^{12\times 20}$ be a CV matrix as in Definition \ref{def51} with the parameter set $\Psi_{\mathbb{CV}}=[\{c_i\}_{i=1}^{12},\{d_i\}_{i=1}^8]^{\top}\in \mathbb{R}^{20},$ where
      \begin{align}
        \left\{ \begin{array}{lcl}
      c_i=\frac{i}{20}, \quad i=1:12,&  \\
       d_j=\frac{j+4}{50},\quad j=1:8.
    \end{array}\right.
    \end{align}
     	For the MNLS solution, we generate a random vector $b\in \R^{12}$ in MATLAB by the command $randn$.  The computed results for the bounds of structured and unstructured CNs are listed in Table \ref{Tab1}. We observed that the bounds for the structured CNs are less than an order of $3$ or $4$ compared to the unstructured case.
  
  \end{example}
    
\begin{example}\label{NEW1}
   Let $M\in \R^{m\times n}$ be a CV matrix as in Definition \ref{def51} with the parameter set $\Psi_{\mathbb{CV}}=[c=\{c_i\}_{i=1}^{m},d=\{d_i\}_{i=1}^l]^{\top}\in \mathbb{R}^{m+l},$ where $c=0.5\, randn(m,1)$ and $d= randn(l,1),$ generate  by the command $randn$ in MATLAB. For the MNLS solution, we take $b\in \R^{m}$. We choose $n=20$ and $l=5.$
    \begin{table}[ht!]
		\begin{center}
		\caption{Comparison between upper bounds of structured and unstructured CNs for $M^{\dagger}(\Psi_{\mathbb{CV}})$ and $M^{\dagger}(\Psi_{\mathbb{CV}},b)$   for Example \ref{NEW1}.}%
			\resizebox{16cm}{!}{
				\begin{tabular}{@{}ccccccccc@{}}
					\toprule
				$m$	&	$\tilde{\M}^{\dagger}(M)$ & $\tilde{\M}^{\dagger}(M(\Psi_{\mathbb{CV}}))$&  $\tilde{\C}^{\dagger}(M)$ & $\tilde{\C}^{\dagger}(M(\Psi_{\mathbb{CV}}))$&$\tilde{\M}^{\dagger}(M,b)$                         & $\tilde{\M}^{\dagger}(M(\Psi_{\mathbb{CV}}),b)$ & $\tilde{\C}^{\dagger}(M,b)$ & $\tilde{\C}^{\dagger}(M(\Psi_{\mathbb{CV}}),b)$\\
    \midrule
   $100$& $5.2481e+04$&$8.5293e+01$&	$9.0595e+07$	&$2.8669e+05$&	$4.4045e+04$	&$1.5229e+02$	&$5.5927e+04$&	$4.3270e+02$	\\
    \midrule
  $125$ &  $3.4639e+04$&	$6.4288e+01$&	$6.1046e+07$&	$2.5378e+05$&	$7.5900e+04$&	$2.3810e+02$	&$1.1277e+06$&	$5.4225e+03$\\
    \midrule
   $150$&$ 1.9732e+04$&	$5.7075e+01$&	$2.9498e+07$&	$1.0762e+05$&	$1.0708e+04$&	$6.2211e+01$&	$3.0915e+04$&	$2.8760e+02$\\
    \midrule      
 $175$ &    $8.6870e+04$&	$8.6711e+01$&	$3.1972e+07$&	$1.1023e+05$&	$6.6658e+04$&	$2.1889e+02$&	$5.5834e+05$	&$7.4617e+03$\\
					\midrule
				$200$&	$1.4574e+04$	&$8.7096e+01$&	$8.9156e+07$	&$3.0088e+05$	&$8.5138e+03$& $7.8091e+01$	& $3.6041e+05$	&$4.3147e+03$  \\
					\toprule
			\end{tabular}}%
   
   \label{TABLE3}
		\end{center}
	\end{table}
The computed results for the bounds of structured and unstructured CNs are listed in Table \ref{TABLE3} for different values of $m$ ranging from $100$ to $200$ with step size $25$.  We observed that the structured CNs have significantly tighter bounds, almost $2$ or $3$ order smaller than those of the unstructured case for large matrices.
\end{example}

\begin{example}\label{ex2}
	We consider a random $\{1,1\}$-QS matrix $M.$ For that, we first fixed $n =5$ and use command \emph{randn} to generate the vectors:
	$$\a\in \R^{n-1}, \, \, \b\in \R^{n-1},\, \, \e\in \R^{n-2}, \, \, {\bf d} \in \R^n, \,\, \f\in \R^{n-1},\, \, \g\in \R^{n-2} \, \text{and} \, \, \h\in \R^{n-1}. $$
	Using these vectors, we computed the $\{1,1\}$-QS matrix defined as \eqref{eq51}. In Table \ref{tab3}, we see that for $M^{\dagger}(\Psi_{\mathbb{QS}})$ the bounds for structured CNs corresponding to $\Psi_{\mathbb{QS}}$ and structured effective CNs are substantially smaller than the unstructured ones. This validates the results of Proposition \ref{pr53}, Theorem \ref{Th9} and Proposition \ref{PR56}. 
	\begin{table}[h]
		\begin{center}
			\tbl{{ Comparison between upper bounds of unstructured, structured and effective CNs for the M-P inverse of $\{1,1\}$-QS matrices for Example \ref{ex2}.}}%
   {
				\begin{tabular}{@{}cccccc@{}}
					\toprule
					$\tilde{\M}^{\dagger}(M)$ & $\tilde{\M}^{\dagger}(M(\Psi_{\mathbb{QS}}))$&  $\tilde{\M}^{\dagger}_{f}(M(\Psi_{\mathbb{QS}}))$&$\tilde{\C}^{\dagger}(M)$ & $\tilde{\C}^{\dagger}(M(\Psi_{\mathbb{QS}}))$ &  $\tilde{\C}^{\dagger}_{f}(M(\Psi_{\mathbb{QS}}))$\\[1ex]
					\midrule
					$817.7856$ &  $3$ &$ 2$ &$ 5.3877e+11$ & $2.4189e+05$ & $1.5682e+05$ \\[1ex]
					\toprule
			\end{tabular}}
   \label{tab3}
		\end{center}
	\end{table}
	Next, for the MNLS solution, we generate a random vector $b\in \R^n$ by the command \emph{randn}. Table \ref{tab4} illustrates that the results obtained for the computed bounds for structured CNs, effective CNs for MNLS solution are much smaller than those for unstructured ones and consistent with Proposition \ref{pr54}.
	\begin{table}[h]
		\begin{center}
			\caption{{ Comparison between upper bounds of unstructured, structured and effective CNs for the MNLS solution for $\{1,1\}$-QS matrices for Example \ref{ex2}.}}\label{tab4}%
			\resizebox{12cm}{!}{
				\begin{tabular}{@{}cccccc@{}}
					\toprule
				$\tilde{\M}^{\dagger}(M,b)$ & $\tilde{\M}^{\dagger}(M(\Psi_{\mathbb{QS}}),b)$&  $\tilde{\M}^{\dagger}_{f}(M(\Psi_{\mathbb{QS}}),b)$&$\tilde{\C}^{\dagger}(M,b)$ & $\tilde{\C}^{\dagger}(M(\Psi_{\mathbb{QS}}),b)$&  $\tilde{\C}^{\dagger}_{f}(M(\Psi_{\mathbb{QS}}),b)$\\ [1ex]
					\midrule
					$1.7838e+03$ & $4.0024$ & $3.0016$&$2.3669e+03$ & $8.8136$ & $6.5184$\\[1ex]
					\bottomrule
			\end{tabular}}
		\end{center}
	\end{table}
\end{example}
\begin{example}\label{ex3}
	We consider several $\{1,1\}$-QS matrices of different orders. In fact, we choose $n=5,$ $n=7$ and $n=10.$  We generate the random vectors $\a,\b, \e,{\bf d}, \f, \g,$ and $\h$ by the command \emph{randn} in MATLAB as in the Example \ref{ex2}. After generating these vectors,  we take following scaling $$ \a=a*10^k,\, \, \e=e*10^k, \, \, \h=h*10^k, $$
	where $ k\in \{-1,-2,0,1,2,3\}$ to get unbalanced lower and upper right corner.\\
	In Table \ref{tab5}, we compare $\tilde{\M}^{\dagger}_{f}(M(\Psi_{\mathbb{QS}}))$ and $\tilde{\C}^{\dagger}_{f}(M(\Psi_{\mathbb{QS}}))$ with $\tilde{\M}^{\dagger}(M)$ and $\tilde{\C}^{\dagger}(M),$ respectively, for the M-P inverse with different values of $n.$ 
	\begin{table}[ht!]
		\centering
		\caption{Comparison between the upper bounds of  unstructured, structured and  effective CNs for the M-P inverse and the MNLS solution  of $\{1,1\}$-QS matrices  for Example \ref{ex3}.}\label{tab5}
		\resizebox{15cm}{!}{
			\begin{tabular}{@{}cccccccccc@{}}
				\toprule
				$n$ & $\tilde{\M}^{\dagger}(M)$&  $\tilde{\M}^{\dagger}_{f}(M(\Psi_{\mathbb{QS}}))$& $\tilde{\C}^{\dagger}(M)$ & $\tilde{\C}^{\dagger}_{f}(M(\Psi_{\mathbb{QS}}))$&$\tilde{\M}^{\dagger}(M,b)$&  $\tilde{\M}^{\dagger}_{f}(M(\Psi_{\mathbb{QS}}),b)$& $\tilde{\C}^{\dagger}(M,b)$ & $\tilde{\C}^{\dagger}_{f}(M(\Psi_{\mathbb{QS}}),b)$\\ [1ex]
				\midrule
				5&	1.4330e+04& 2 & 2.1690e+{12} &  910&	9.333e+04 & 3.0030 &1.1910e+05 &  4.5451\\
				\midrule
				7& 5.9139e+03 & 3.0246 & 9.2340e+08 & 74.2974& 6.5311e+04 & 4.6655 & 2.0255e+ 04 & 6.7386  \\
				\midrule
				10  & 7.3293e+04&  2.0101&  1.5858e+10& 94.0499 & 6.3988e+04&  1.0127&  2.5381e+05& 11.1271 \\ [1ex]
				\toprule
		\end{tabular}}
	\end{table}
	For the MNLS solution, we generate a random vector $b\in \R^n$ for each choice of $n.$ The computed bounds for the CNs are listed in Table \ref{tab5}. These results demonstrate the reliability of proposed CNs.
\end{example}
\begin{example}\label{ex4}
	In this example, we compare structured MCN and CCN  with respect to QS representation and GV representation through tangent, structured effective CNs with their unstructured ones for M-P inverse and MNLS solution. On account of these, we generate the random vectors $t\in \R^{n-2},u\in \R^{n-1},v\in \R^{n-1},w\in \R^{n-1}$ in MATLAB by the comand {\textit{randn}}  for the parameter set $\Psi_{\mathcal{GV}}.$ We set $\d={\bf 0}\in \R^n$ and rescale the vector $v$ as $v(1)=0$ and $v(n-1)=10^2.$ Then, we compute the $\{1,1\}$-QS matrix $M$ as in Definition \ref{def52}. For different values of $n,$ we generate $100$ rank deficient $\{1,1\}$-QS matrices. We use the formulae provided in Theorem \ref{Th53} to compute the upper bounds $\tilde{\M}^{\dagger}(M(\Psi_{\mathcal{GV}})) $ and $\tilde{\C}^{\dagger}(M(\Psi_{\mathcal{GV}})).$ Again, we use the formulae for $\tilde{\M}^{\dagger}_{f}(M(\Psi_{\mathbb{QS}}))$ and $\tilde{\C}^{\dagger}_{f}(M(\Psi_{\mathbb{QS}}))$ presented as in Definition \ref{def54}, and for $\tilde{\M}^{\dagger}(M(\Psi_{\mathbb{QS}}))$ and $\tilde{\C}^{\dagger}(M(\Psi_{\mathbb{QS}}))$ presented as in Theorem \ref{Th51}. 
	
	We computed the above values for $100$ randomly generated $\{1,1\}$-QS matrices for $n=30, 40,50$ and $60.$ In Table \ref{tab7},  average values of each upper bound of the CNs for the M-P inverse of these $\{1,1\}$-QS matrices are listed.
	
	\begin{table}[ht!]
		\centering
		\caption{Comparison between upper bounds of unstructured, structured and  effective CNs  for M-P inverse of $\{1,1\}$-QS matrices  for Example \ref{ex4}.}\label{tab7}
		\resizebox{16cm}{!}{
			\begin{tabular}{@{}cccccccccc@{}}
				\toprule
				mean	&	$n$ & $\tilde{\M}^{\dagger}(M)$& $\tilde{\M}^{\dagger}(M(\Psi_{\mathcal{GV}}))$ & $\tilde{\M}^{\dagger}(M(\Psi_{\mathbb{QS}}))$& $\tilde{\M}^{\dagger}_{f}(M(\Psi_{\mathbb{QS}}))$& $\tilde{\C}^{\dagger}(M)$ & $\tilde{\C}^{\dagger}(M(\Psi_{\mathcal{GV}})) $ &$\tilde{\C}^{\dagger}(M(\Psi_{\mathbb{QS}}))$& $\tilde{\C}^{\dagger}_{f}(M(\Psi_{\mathbb{QS}}))$\\ [1ex]
				\midrule
				& $30$ & $1.0667e+02$ & $7.4873e+01$ & $1.2388e+02$ & $8.9215e+01$ & $2.3780e+05$ & $1.2730e+04$ & $2.2102e+04$ & $1.5815e+04$ \\[1ex]
				\midrule
				& $40$ & $1.1429e+02$ & $7.4757e+01$ & $1.2356e+02$ & $8.9427e+01$ & $6.1267e+08$ & $1.4452e+05$ &$ 2.3753e+05$ & $1.7285e+05$ \\[1ex]
				\midrule
				& $50 $&$ 1.6711e+02$ & $1.0323e+02$ & $1.7182e+02$ & $1.2296e+02$ &$ 6.9215e+06$ & $1.6133e+05$ & $3.2430e+05$ & $2.1980e+05$ \\[1ex]
				\midrule
				& $60$ & $3.2135e+02$ & $1.8140e+02$ & $2.9452e+02$ & $2.1359e+02$ & $2.7856e+07$ & $3.0500e+05 $& $5.2242e+05$ & $3.9791e+05$ \\[1ex]
				\toprule
			\end{tabular}
		}
	\end{table}
	\begin{table}[ht!]
		\centering
		\caption{Comparison between upper bounds of unstructured, structured and effective CNs for the MNLS solution for $\{1,1\}$-QS matrices for Example \ref{ex4}.}
		\label{tab8}
		\resizebox{16cm}{!}{
			\begin{tabular}{@{}cccccccccc@{}}
				\toprule
				mean	&	$n$ & $\tilde{\M}^{\dagger}(M,b)$& $\tilde{\M}^{\dagger}(M(\Psi_{\mathcal{GV}}),b)$ & $\tilde{\M}^{\dagger}(M(\Psi_{\mathbb{QS}}),b)$& $\tilde{\M}^{\dagger}_{f}(M(\Psi_{\mathbb{QS}}),b)$& $\tilde{\C}^{\dagger}(M,b)$ & $\tilde{\C}^{\dagger}(M(\Psi_{\mathcal{GV}}),b) $ &$\tilde{\C}^{\dagger}(M(\Psi_{\mathbb{QS}}),b)$& $\tilde{\C}^{\dagger}_{f}(M(\Psi_{\mathbb{QS}}),b)$\\ [1ex]
				\midrule
				& $30 $& $1.1278e+02$ & $ 7.8851e+01$ & $ 2.6113e+02$ & $ 9.4414e+01$ & $ 4.6667e+03$ & $ 1.7445e+03$ & $ 2.2102e+04 $& $ 2.1646e+03$ \\[1ex]
				\midrule
				& $40 $& $1.1990e+02$ & $7.7823e+01$ & $2.4226e+02$ & $9.3005e+01$ & $7.3841e+03$ &$ 4.1774e+03$ & $2.3753e+05$ & $5.2069e+03$ \\ [1ex]
				\midrule
				& $50$ & $1.6512e+02$ & $1.0080e+02$ & $3.2184e+02$ & $1.2072e+02$ & $9.2143e+03$ & $4.2351e+03$ & $3.2430e+05 $&$ 5.3584e+03$ \\ [1ex]
				\midrule
				& $60$ &$ 3.3149e+02$ & $1.8715e+02$ &$ 7.6179e+02$ & $2.2120e+02$ & $7.0839e+04$ & $5.9482e+04 $& $5.2242e+05$ & $7.2054e+04$ \\ [1ex]
				\toprule
			\end{tabular}
		}
	\end{table}
 Next, we generate a random vector $b\in \R^n,$ for each value of $n,$ and upper bounds for the structured CNs with respect to QS representation and  GV representation through tangent, structured effective CNs and unstructured CNs for the MNLS solution are listed in Table \ref{tab8}.
 
 The computed results in Tables \ref{tab7} and \ref{tab8} shows the  consistency of  Propositions \ref{pr53}--\ref{PR56} and Theorem \ref{Th9}. We can observe that the upper bounds of CNs for GV representation through tangent give more reliable bounds compared to the other CNs. 
\end{example}

\begin{example}\label{exam67}
   In this example, we consider  $\{1,1\}$-QS matrices of different orders. To construct  $\{1,1\}$-QS matrices using the formula given in Definition \ref{def52}, we generate the random vectors $t\in \R^{n-2},u\in \R^{n-1},v\in \R^{n-1},w\in \R^{n-1}$ as in Example \ref{ex4}. Moreover, we generate $\d=randn(n)\in \R^n$ and set $\d(1)=0$. Further, we rescale $v$ by setting $v(1)=0$ and $v(n-1)=1.$ For MNLS solution, we choose $b=randn(n)\in \R^{n}.$
   \begin{table}[ht!]
		\centering
		\caption{ Comparison between upper bounds of unstructured, structured and  effective CNs  for M-P inverse of $\{1,1\}$-QS matrices  for Example \ref{exam67}.}
   \label{tab9}
		\resizebox{16cm}{!}{
			\begin{tabular}{@{}ccccccccc@{}}
				\toprule
					$n$ & $\tilde{\M}^{\dagger}(M)$& $\tilde{\M}^{\dagger}(M(\Psi_{\mathcal{GV}}))$ & $\tilde{\M}^{\dagger}(M(\Psi_{\mathbb{QS}}))$& $\tilde{\M}^{\dagger}_{f}(M(\Psi_{\mathbb{QS}}))$& $\tilde{\C}^{\dagger}(M)$ & $\tilde{\C}^{\dagger}(M(\Psi_{\mathcal{GV}})) $ &$\tilde{\C}^{\dagger}(M(\Psi_{\mathbb{QS}}))$& $\tilde{\C}^{\dagger}_{f}(M(\Psi_{\mathbb{QS}}))$\\ [1ex]
				\midrule
				$100$ & $4.3964e+02$&	$3.6260e+02$&	$5.3677e+02$&	$3.9861e+02$&	$4.5243e+07$&	$2.4837e+04$&	$3.6640e+04$&	$2.7217e+04$ \\[1ex]
				\midrule
				$150$ & $4.1692e+01$ &	$3.6846e+01$ & $	5.2059e+01$&	$4.3613e+01$	& $ 1.0494e+09$ &	$7.5769e+06$ &	$1.1606e+07$ &	$9.2937e+06$\\[1ex]
				\midrule
				$200 $&$3.3461e+02$&	$1.9504e+02$	&$2.9550e+02$&	$2.4237e+02$	&$ 1.5109e+07$&	$1.0192e+05$ & $	1.9086e+06$ & $1.4250e+05$\\[1ex]
				\midrule
   $250$ &$2.1300e+02$&	$1.3904e+02$&	$2.1856e+02$&	$1.5019e+02$&	$4.8391e+08$&	$2.4495e+06$&	$4.1373e+06$&	$3.1253e+06$\\[1ex]
    \midrule
				$300$ & $2.0939e+02$&$	9.8454e+01$ &	$1.6871e+02$	& $ 1.1458e+02$ &	$1.1827e+09$&	$1.6759e+06$&	$1.7875e+07$	& $1.7776e+06$ \\[1ex]
				\toprule
			\end{tabular}
		}
	\end{table}
 \begin{table}[h!]
		\centering
		\caption{ Comparison between upper bounds of unstructured, structured and effective CNs for the MNLS solution for $\{1,1\}$-QS matrices for Example \ref{exam67}.} 
		\label{tab10}
		\resizebox{16cm}{!}{
			\begin{tabular}{@{}ccccccccc@{}}
				\toprule
					$n$ & $\tilde{\M}^{\dagger}(M,b)$& $\tilde{\M}^{\dagger}(M(\Psi_{\mathcal{GV}}),b)$ & $\tilde{\M}^{\dagger}(M(\Psi_{\mathbb{QS}}),b)$& $\tilde{\M}^{\dagger}_{f}(M(\Psi_{\mathbb{QS}}),b)$& $\tilde{\C}^{\dagger}(M,b)$ & $\tilde{\C}^{\dagger}(M(\Psi_{\mathcal{GV}}),b) $ &$\tilde{\C}^{\dagger}(M(\Psi_{\mathbb{QS}}),b)$& $\tilde{\C}^{\dagger}_{f}(M(\Psi_{\mathbb{QS}}),b)$\\ [1ex]
				\midrule
				$100 $& $4.3739e+02$&	$3.6376e+02$&	$2.0884e+02$	&$4.0061e+02$&	$6.1525e+0$3&	$4.8648e+02$& $	2.5191e+03$&	$5.4354e+03	$ \\[1ex]
				\midrule
				 $150 $& $4.7556e+01$	&$4.3981e+01$&	$7.0553e+01$	&$5.1641e+01$	&$5.3904e+03$&	$3.3664e+02$&	$1.2982e+03$&	$3.8831e+02$ \\ [1ex]
				\midrule
				$200$ &$3.4402e+02$& $2.0275e+02$&	$1.0998e+02$&	$2.5110e+02$&	$1.4342e+04$&	$5.3603e+03$&	$1.1052e+04$&	$6.5164e+03$\\ [1ex]
				\midrule
   $250$& $1.0986e+02$&$	1.0489e+02$	&$1.8630e+02$&	$1.1328e+02$&	$8.4676e+04$&	$5.1797e+03$	&$1.0935e+04$&$	6.0154e+03$\\ [1ex]
    \midrule 
				 $300$ &$2.1839e+02$&	$1.0345e+02$&	$2.6652e+02$&	$1.2097e+02$&	$3.4116e+03$	&$1.4340e+03$&	$6.5971e+03$&	$1.7712e+03$\\ [1ex]
				\toprule
			\end{tabular}
		}
	\end{table}
For different values of $n,$ the computed upper bounds for structured CNs with respect to QS representation and GV representation through tangent, structured effective CNs and unstructured CNs for M-P inverse and MNLS solution are reported in Tables \ref{tab9} and \ref{tab10}. These results confirm that our proposed upper bounds are reliable for large matrices, and structured ones are much sharper and smaller than unstructured ones.
\end{example}

\section{Conclusion}\label{s7}
For the M-P inverse and the MNLS solution, we investigated structured MCN and CCN corresponding to a class of parameterized matrices, with each entry as a differentiable function of some real parameters. This framework has been used to derive the upper bounds of structured CNs for CV and $\{1,1\}$-QS matrices. QS representation and the GV representation through tangent are considered for $\{1,1\}$-QS matrices to investigate their structured CNs. It is proved that upper bounds for the structured CNs for GV representation through tangent are always smaller than the QS representation. Numerical examples demonstrate that the proposed structured effective CNs are significantly smaller in most cases. Our methodology can be applied to investigate CNs involving other rank-structured matrices. It would be interesting to generalize our findings to the weighted M-P inverse and weighted LS problems, as well as linear LS problems with multiple right-hand sides. Following that, research progress on the aforementioned topics will be reported elsewhere.
\section*{Acknowledgments}
We gratefully acknowledge the financial assistance (File no. $09/1022(0098)/2020$-EMR-I) provided as a fellowship to Pinki Khatun by the Council of Scientific $\&$ Industrial Research (CSIR), New Delhi, India. The authors are grateful to anonymous reviewers for their constructive comments and suggestions, which have contributed to enhancing the quality of this paper. We extend our sincere gratitude to Prof. Huaian Diao for providing us with the thesis \cite{thesis}.
\section*{Declarations}
{\bf Conflict of interest} The authors declare no competing interests.
\bibliography{reference}
\bibliographystyle{abbrvnat}
\end{document}